\documentclass[a4paper, english, 11 pt]{amsart}
\usepackage{amsaddr} 
\usepackage[utf8]{inputenc}
\usepackage[T1]{fontenc}
\usepackage{textcomp}

\usepackage{amsfonts}
\usepackage{amssymb}
\usepackage{amsmath}
\usepackage{amsthm}
\usepackage{dsfont}

\usepackage{epstopdf}
\usepackage[dvips]{graphicx}
\usepackage{graphicx}
\usepackage{xcolor}
\usepackage{xspace}

\usepackage{relsize}
\usepackage[bottom]{footmisc}
\usepackage[small]{caption}
\everymath{\displaystyle}

\usepackage{enumerate}
\usepackage{multirow}
\usepackage{hyperref}

\usepackage[
backend=biber,
style=alphabetic
]{biblatex}
\newtheorem{thm}{Theorem}[section]
\newtheorem{prop}[thm]{Proposition}
\newtheorem{lem}[thm]{Lemma}
\newtheorem{cor}[thm]{Corollary}
\theoremstyle{definition}
\newtheorem{defi}[thm]{Definition}
\newtheorem{ass}[thm]{Assumptions}

\theoremstyle{remark}
\newtheorem{rem}[thm]{Remark}
\numberwithin{equation}{section}

\let\cal\mathcal
\let\rm\mathrm

\let\bf\mathbf
\let\ov\overline

\let\dg\dagger

\newcommand{\N}{\mathbb{N}}

\newcommand{\R}{\mathbb{R}}
\newcommand{\C}{\mathbb{C}}
\newcommand{\K}{\mathbb{K}}
\newcommand{\Ss}{\mathbb{S}}

\newcommand{\Pp}{\mathbb{P}} 
\newcommand{\E}{\mathbb{E}} 
\DeclareMathOperator{\Var}{Var} 
\DeclareMathOperator {\iid} {iid}
\newcommand{\mmp}{\mathrm{MP}} 

\newcommand{\Tr}{\mathrm{Tr}} 
\newcommand{\Sp}{\mathrm{Sp}} 

\newcommand{\verti}[1]{{\left\vert #1 \right\vert}}
\newcommand{\vertii}[1]{{\left\vert\kern-0.25ex\left\vert #1 
    \right\vert\kern-0.25ex\right\vert}}
\newcommand{\vertiii}[1]{{\left\vert\kern-0.25ex\left\vert\kern-0.25ex\left\vert #1 \right\vert\kern-0.25ex\right\vert\kern-0.25ex\right\vert}}
\newcommand{\vertif}[1]{{\left\vert\kern-0.25ex\left\vert #1 
    \right\vert\kern-0.25ex\right\vert}_F}
\newcommand{\vertim}[1]{{\left\vert\kern-0.25ex\left\vert #1 
    \right\vert\kern-0.25ex\right\vert}_{\max}}
\newcommand{\pt}[1]{{\left( #1 \right)}}
\newcommand{\br}[1]{{\left[ #1 \right]}}

\newcommand{\half}{\frac 1 2} 

\newcommand{\lint}{[\![}
\newcommand{\rint}{]\!]}

\begin{document}

\title[Deterministic equivalent of sample covariance matrices]{Quantitative deterministic equivalent of sample covariance matrices with a general dependence structure}

\author{Clément Chouard}\thanks{\textbf{Acknowledgments:} This work received support from the University Research School EUR-MINT (reference number ANR-18-EURE-0023).\newline The author would like to thank the team of Statistical Field Theory at EPFL Lausanne, and especially Franck Gabriel and Berfin Simsek, for their hearty welcome and the interesting discussions that made the application to Kernel Methods (Section 9) happen.}
\address{Institut de Mathématiques de Toulouse; UMR5219\\ Université de Toulouse; CNRS; UPS, F-31062 Toulouse, France}
\email{clement.chouard@math.univ-toulouse.fr}

\maketitle

\begin{abstract} We study sample covariance matrices arising from rectangular random matrices with i.i.d. columns. It was previously known that the resolvent of these matrices admits a deterministic equivalent when the spectral parameter stays bounded away from the real axis. We extend this work by proving quantitative bounds involving both the dimensions and the spectral parameter, in particular allowing it to get closer to the real positive semi-line.

As applications, we obtain a new bound for the convergence in Kolmogorov distance of the empirical spectral distributions of these general models. We also apply our framework to the problem of regularization of Random Features models in Machine Learning without  Gaussian hypothesis.
\end{abstract}

\tableofcontents

\newpage

\section{Introduction}

Since the pioneering works in Statistics of \cite{wishart1928}, an object of great importance to consider is the sample covariance matrix $K = \frac 1 n X X^\top$ where $X$ is a rectangular $p\times n$ random matrix made of $n$ independent columns. One is particularly interested in the asymptotic behaviour of the empirical spectral distribution (ESD) $\mu_K = \frac 1 p \sum_{\lambda \rm{~eigenvalue~of~} K} \delta_\lambda$ when $n$ and $p$ go to infinity, and the ratio $\frac p n$ converges to a fixed constant $\gamma > 0$. 
When all the entries of $X$ are independent and identically distributed (i.i.d.), it can be shown (\cite{MP67}, \cite{silverstein1989eigenvectors}) that $\mu_K$ converges almost surely to a deterministic probability measure $\mmp (\gamma)$, now called the Marčenko-Pastur distribution with shape parameter $\gamma$. From this article on, many works have provided generalizations beyond the i.i.d. case. This is the case for instance in \cite{silverstein1995analysis}, where the authors investigated the linearly dependent case. Namely when $X = \Sigma^{\frac 1 2} Y$ with $\Sigma$ a real symmetric matrix and $Y$ filled with i.i.d. entries, $\mu_K$ converges a.s. to a deterministic probability measure that can be described as the free multiplicative convolution $\mmp(\gamma) \boxtimes \mu_\infty $, where $\mu_\infty$ is the limiting ESD of $\Sigma$. Under the sole assumption that the columns of $X$ are independent, \cite{bai2008large} proved that $\mu_K$ still converges a.s. to $\mmp(\gamma) \boxtimes \mu_\infty $ where $\mu_\infty$ is the limiting ESD of $\Sigma:=\E[K]$.

A key object in Random Matrix Theory is the resolvent matrix $\cal G_K = \pt{K - z I_p}^{-1}$, where $z$ is a complex number called spectral argument of $\cal G_K$. This matrix is related to the Stieltjes transform $g_K = \int_{\R} \frac{1}{t - z} \mu_K(dt)$ via the identity $g_K= \frac 1 p \Tr\pt{\cal G_K} $. This integral transform is a powerful tool that characterizes the law and the weak convergence of probability measures. For more  advanced applications to the distribution of eigenvalues and eigenvectors, control of the Stieltjes tranform is not enough and one needs control of the whole resolvent matrix $\cal G_K$. This was investigated in  \cite{bloemendal2014isotropic} in the i.i.d. case, establishing that $\cal G_K$ is close to $g_K I_p$ with quantitative bounds involving the dimensions and the spectral parameter $z$. This analysis was later carried to the linearly dependent case (\cite{knowles2017anisotropic}), showing that $\cal G_K$ is close to a deterministic matrix $\bf G$ which is not a multiple of the identity matrix in general. Following the terminology of \cite{hachem2007deterministic}, we call the matrix $\bf G$ a deterministic equivalent of $\cal G_K$. In the most general case dealing with independent columns, \cite{LC21} found a similar deterministic equivalent. Remarkably, they consider columns with different distributions that were not considered in the previous literature.

This last article however did not allow the spectral argument to get closer to the real axis with quantitative bounds on $z$. We complete it by quantifying the convergence towards the deterministic equivalent when the underlying random matrix has i.i.d. columns. Our result encompasses two different settings: when $z$ is a sequence of complex numbers with positive imaginary parts that do not vanish too quickly, and when $z$ is a sequence of negative real numbers. In both cases, we prove the existence of polynomial bounds in $n$, $\Im(z)$ and $|z|$, uniform  in the sense that the bounds only involve a few well-identified parameters and not the whole law of the matrices. Our results also have the additional benefit of explaining the appearance of free convolution operations in the deterministic equivalent. Although not optimal compared to numerical simulations or previous similar statements in the linearly dependent case (\cite{knowles2017anisotropic}), our estimates still allow for important  new applications. 

In random matrix theory, a classic question arising after the existence of a weak limit for the ESD of a given model is to know whether this convergence can be quantified using a distance between probability distributions. \cite{bai2008convergence} first tackled the problem of convergence rate in the i.i.d. case, a  result that was later improved in \cite{gotze2010rate}. In the linearly dependent case, a similar result was proved in \cite{bai2012convergence}.
We complete this work by proposing a bound in  Kolmogorov distance in the most general case of matrices with independent columns.

Such a deterministic equivalent can also prove itself useful for Machine Learning applications. In the Method of Random Features associated to Kernel Ridge Regression, some natural predictors are systematically biased.  Under Gaussian hypothesis, the authors of \cite{jacot2020implicit} proved that this issue could be solved by replacing the ridge parameter of the regression by a different parameter called effective ridge. We are able to extend this result beyond the Gaussian paradigm, showing in a sense the universality of the phenomenon regardless to the laws of the random features.

This paper uses a wide variety of techniques. Concentration of measure theory in particular is a powerful tool to deal with large dimensional random objects. A core result in \cite{ledoux2001concentration} highlights the importance of a class of concentrated vectors compatible with Lipschitz mappings. Logically flowing from this idea, we adopt the recent formalism of \cite{LC20} as it is particularly well adapted to deal with random matrices and their transformations. Along with the classic concentration properties, we use generalizations of the so-called Hanson-Wright inequalities (\cite{vu2015random}, \cite{adamczak2015}) which state the concentration of quadratic forms involving random vectors and matrices.

This paper also uses extensively the theory behind resolvent matrices. In addition to the usual sample covariance resolvent, we introduce two companion objects: the Leave-One-Out (LOO) resolvent and the co-resolvent. We provide a  systematic analysis of the analytical and algebraic joint properties of these objects. In conjunction with concentration, this work allows to streamline some of the technical arguments, in particular we do not resort anymore to concentration zone arguments. It also helps to better understand the properties of the deterministic equivalent and the links with free probability theory.

Finally to convert bounds involving Stieltjes transforms to bounds in Kolmogorov distance, a general methodology was developed in \cite{banna2020holder}. We adapt these techniques to work with looser requirements adapted to our needs.

\newpage

\newpage

\section{Setting}

\subsection{Notations and definitions} The set of matrices with $p$ lines, $n$ columns, and entries belonging to a set $\K$ is denoted as $\K^{p \times n}$. We will use the following norms for vectors and matrices: $\vertii{\cdot}$ the Euclidean norm, $\vertif{\cdot}$ the Frobenius norm, and $\vertiii{\cdot}$ the spectral norm.

Given $M \in \C^{ p \times p}$, we denote respectively $M^\top$ its transpose, $M^\dg$ its transconjugate and $\Tr(M)=\sum_{1 \leq i \leq p} M_{ii}$ its trace. If $M$ is real and diagonalizable we denote by $\Sp M$ its spectrum and $\mu_M = \frac 1 p \sum_{\lambda \in \Sp M} \delta_\lambda$ its empirical spectral distribution (ESD). Let us define the sphere $\Ss^{p-1}=\{ \bf u \in \R^p$ such that $\vertii{\bf u}=1\}$. If $\bf u \in \Ss^{p-1}$, we also define the eigenvector empirical spectral distribution (VESD) of $M$ in the direction $\bf u$ as $\mu_{M,\bf u} = \sum_{\lambda \in \Sp M} \pt{\bf u^\top \bf v_\lambda}^2 \delta_\lambda$, where $\bf v_\lambda \in \Ss^{p-1}$ are normalized eigenvectors associated to the eigenvalues $\lambda$.

Let us consider a spectral parameter $z$, either belonging to $\R^{*-}:= (-\infty,0)$ or $\C^{+}:= \{ z \in \C$ such that $\Im(z) > 0 \}$. If $M$ is symmetric positive semi-definite, then $\Sp M \subset \R^+$ and we can define its resolvent $\cal G_M(z) = (M-zI_p)^{-1}$ and its Stieltjes transform $g_M(z) = \frac 1 p \Tr \pt{\cal G_M(z)}$.

Given a probability distribution $\nu$ supported on $\R^+$, its Stieltjes transform is $g_\nu(z) = \int_\R \frac {\nu(dt)} {t-z}  $ where $z \in \R^{*-}$ or $z \in \C^+$. The Stieltjes transform of a matrix is the same as the Stieltjes transform of its ESD: $g_{\mu_M} = g_M$. There is a similar link for the VESD's: $g_{\mu_{M,\bf u}} = \bf u^\top \cal G_M \bf u$.

We let $\cal F_\nu$ be the cumulative distribution function (CDF) of $\nu$, and $\Delta(\nu,\mu) = \sup_{t \in \R} \verti{\cal F_\nu(t) - \cal F_\mu(t)}$ be the Kolmogorov distance between two measures $\nu$ and $\mu$.

We denote by $\boxtimes$ the multiplicative free convolution of measures (\cite{bercovici1993free}). If $\mmp(\gamma)$ denotes a Marčenko-Pastur distribution with shape parameter $\gamma$, the distribution $\mmp(\gamma) \boxtimes \mu$ may be defined by its Stieltjes transform $g$ which is the unique solution of the self-consistent equation \cite{MP67}:
\begin{align*} g(z) = \int_\R \frac 1 {( 1 - \gamma - \gamma z g(z) ) t - z } \mu(dt)
\end{align*}

For a better readability, we will sometimes omit indices $n$ and parameters $z$ in our notations.

\subsection{Main results}

Let $X \in \R^{p_n \times n}$ be a sequence of random matrices. The associated sample covariance matrix is  $K = n^{-1} X X^\top$. We set $\Sigma= \E[K]$, and  we also define the resolvent matrix $\cal G_K(z) = \pt{K-zI_p}^{-1}$.

\begin{ass} \label{Ass}
\begin{enumerate}
    \item  The columns of $X$ are i.i.d. sampled from the distribution of a random vector $x$, such that $\vertii{\E[x]}$ and $\vertiii{\Sigma}=\vertiii{\E[x x^\top]}$ are bounded.
      \item The sequence $X$ is $\propto \cal E_2(1)$ concentrated with respect to the Frobenius norm, in the sense that all $1$-Lipschitz functions of $X$ satisfy a concentration inequality in $e^{-(\cdot)^2}$ (see Definition \ref{DefConcentration} below).
\item $\gamma_n:= \frac {p_n} n$ is bounded from above and from below: there is a constant $C>0$ such that $C^{-1} \leq \gamma_n \leq C$.
\end{enumerate}
\end{ass}

We start with a simple proposition establishing the concentration of the resolvent of the sample covariance matrix $K$:

\begin{prop}\label{ConcentrationGK} 
\begin{enumerate}
    \item If $z_n \in \R^{*-}$ is a sequence of real spectral arguments, and $\tau := \frac 1 { |z|^{\frac 3 2}} $, then $\cal G_K(z) \propto \cal E_2\pt{\frac \tau  {n^{\half}}}$.
    \item If $z_n \in \C^+$ is a sequence of complex spectral arguments, and $\tau := \frac {|z|^\half} {\Im(z)^2 }$, then $\cal G_K(z) \propto \cal E_2\pt{\frac \tau  {n^{\half}}}$.
\end{enumerate}
\end{prop}

In order to state the main result of this article, we need to introduce the probability measures $ \nu_n = \mmp\pt{\gamma_n } \boxtimes \mu_{\Sigma}$ and $\check \nu_n = (1-\gamma_n ) \delta_0 + \gamma_n  \nu_n$. We also build the sequence of matrices: 
\[\bf G(z) = \pt{- z g_{\check \nu}(z) \Sigma - z I_p}^{-1},\]where $z \in \R^{*-}$ or $z \in \C^+$.

\begin{thm}\label{MainThm}
\begin{enumerate}  \item If $z_n \in \R^{*-}$ is a sequence of real spectral arguments, such that $|z|$ is bounded and $|z|^{-7} \leq O(n)$, then $\vertif{\E[\cal G_K(z)] - \bf G(z) } \leq O\pt{\frac \kappa {n^\half} }$, where $\kappa := \frac 1 {|z|^{\frac {11} 2}}$.
    \item If $z_n \in \C^+$ is a sequence of complex spectral arguments, such that $\Im(z)$ is bounded and ${\Im(z)^{-16} |z|^{ 7 }} =o(n)$, then $\vertif{\E[\cal G_K(z)] - \bf G(z) } \leq O\pt{\frac \kappa {n^\half} }$, where $\kappa := \frac{ |z|^{\frac 5 2}}{\Im(z)^9}$.
\end{enumerate}

These result are uniform in the sense that the implicit constants in the $O(\cdot)$ notations only depend on the  constants in the hypothesis chosen for $X$, $\gamma_n$, $\Sigma$, $\E[x]$ and $z$.
\end{thm}

This result can be combined with the concentration of the resolvent to obtain a so-called deterministic equivalent of the resolvent:

\begin{prop}\label{DetEq}  With the same assumptions and notations as in Theorem \ref{MainThm}, $\bf G$ is a deterministic equivalent for $\cal G_K$. More precisely, uniformly in any deterministic matrix $A \in \R^{p \times p}$, we have:
\begin{align*} \verti{\Tr\pt{\cal G_K(z) A - \bf G(z) A}} &\leq \vertif{A} \, O\pt{\frac {\kappa (\log n)^\half} {n^\half} }  \rm{\quad a.s. } \\ \rm{and}\quad \Var\br{\Tr\pt{\cal G_K(z) A }} &\leq \vertif{A}^2  \, O\pt{\frac {\kappa^2} n }.
\end{align*} 
\end{prop}

Specifying the matrix $A$ and using the relation $p^{-1} \Tr {\bf G} = g_{\nu}$, which shall be proved in the course of this article, yields the following consequences{~}:

\begin{cor}\label{CorDetEq} With the same assumptions and notations as in Theorem \ref{MainThm}:
\begin{align*} \verti{g_K(z) - g_\nu(z)} &\leq O\pt{\frac {\kappa (\log n)^\half} {n} }  \rm{\quad a.s. } \\ \verti{g_{\mu_{K,\bf u}}(z) - \bf u^\top \bf G(z) \bf u} &\leq O\pt{\frac {\kappa (\log n)^\half} {n^\half} }  \rm{\quad uniformly~in~}\bf u \in \Ss^{p-1} \quad\rm{~a.s.} \\
\max_{1 \leq i,j\leq p} \, \verti{ \cal G_K(z) - \bf G(z) }_{ij} &\leq O\pt{\frac {\kappa (\log n)^\half} {n^\half} }  \rm{\quad a.s. }
\end{align*} 
\end{cor}

\begin{rem}
Similar statements can be made about the variances, for instance $\Var[g_K(z) ] \leq O\pt{\frac{\kappa^2}{n^2}}$.

The Stieltjes transform of the ESD is better concentrated than the Stieltjes transforms of the VESD's. That was to be expected since it corresponds to taking an average over a basis of orthonormal vectors.

Finally all these results admit a transposed version for models of random matrices $X' \in \R^{n \times p}$ having $\iid$ sampled rows  from a random vector with covariance $\Sigma$. In this setting the covariance matrix has to be chosen as $K' = \frac 1 n X'^\top X' $ and the equivalent objects are left unchanged.  
\end{rem}

\subsection{Application to Kolmogorov distances}

Our main results give quantitative bounds for sequences of spectral arguments that get close the real axis. This allows to use general methods that bound the Kolmogorov distance of two distributions given a precise control on their Stieltjes transforms. The consequences are the following:

\begin{thm}\label{ApplicationKolmo} Under the Assumptions \ref{Ass} and the additional hypothesis that $\Sigma$ is invertible and   $\vertiii{\Sigma^{-1}}$ is bounded, we have: 
\[ \Delta(\mu_K, \nu) \leq  O\pt{n^{- \frac 1 {70}}} \rm{\quad a.s. } \]
\end{thm}
\begin{cor}\label{ApplicationKolmoCV} 
If additionally $\mu_\Sigma \to \mu_\infty$ weakly and $\gamma_n \to \gamma_\infty $, then $\mu_K$ converges weakly to $\nu_\infty:= \mmp(\gamma_\infty) \boxtimes \mu_\infty$ a.s. Moreover:
\[ \Delta (\mu_K, \nu_\infty) \leq \Delta(\mu_\Sigma,\mu_\infty) + O\pt{\verti{\gamma_n - \gamma_\infty}} + O\pt{n^{- \frac 1 {70}}} \rm{\quad a.s. }  \] 
\end{cor}

\subsection{Application to Kernel Methods}

This second application to the regularization of Random Feature Models in Machine Learning closely follows the works of \cite{jacot2020implicit}.
Random Feature (RF) models are used as an  efficient
approximation of Kernel Ridge Regression (KRR) methods. In the regime on the dimensions where this approximation has a computational interest, a systematic asymptotic bias appears unfortunately.

Under Gaussian assumptions, the authors of \cite{jacot2020implicit} showed an implicit regularization phenomenon: the bias can be removed by replacing $\lambda$, the ridge parameter used in the regressions, by another parameter $\tilde \lambda$. They called this new parameter effective ridge, and characterized it by means of a self-consistent equation.

Using our deterministic equivalent, we  extend the aforementioned work when the data are not necessarily Gaussian distributed, and we also explain how the effective ridge parameter $\tilde \lambda$ is related to a free multiplicative convolution of measures. We refer to the Section 9 for more details on these results.

\subsection{Organization of the paper}

The bulk of this paper is devoted to the proof of the deterministic equivalent Theorem \ref{MainThm}, which is the content of Sections 5-7.

In Section 3 we present the concentration framework and the key concentration results that we shall use throughout the proofs. In Section 4 we study thoroughly the  properties of the resolvents of sample covariance matrices and two companion objects: the Leave One Out (LOO) resolvent  and the co-resolvent. 

In Section 5 we perform the first step of the proof by establishing an intermediate deterministic equivalent based on algebraic identities. In Section~6 we pass from this first equivalent to a second deterministic equivalent corresponding to our main result. The key ingredient of this step consists in rewriting the problem as fixed point equation implying a functional that is a contraction mapping for a well-chosen metric. The arguments for real and complex spectral parameters differ slightly and are treated separately. In Section 7 we wrap up the proofs of Theorem \ref{MainThm} and the subsequent corollaries.

In Section 8, as an illustration of the power of our quantitative results, we present an application to the convergence in Kolmogorov distance for general models of sample covariance matrices with i.i.d. columns.

In Section 9, we explain how our work is related to the problem of Kernel Ridge Regression (KRR) in machine learning, and we retrieve similar results to those of \cite{jacot2020implicit} under broader hypothesis.
\newpage

\section{Concentration framework}

In this preliminary section we introduce some basic definitions and results to handle concentrated random vectors and matrices. We adopt the recent formalism of Lipschitz concentration of Louart and Couillet as it is particularly well adapted to deal with random matrices and their transformations. We advice the reader to consult their articles (\cite{LC19}, \cite{LC20}, \cite{LC21}) for a more comprehensive coverage of these notions.

\begin{defi} \label{DefConcentration} We say that a sequence of random vectors $X_n$, belonging to normed vector spaces $(E_n,\vertii{\cdot})$, is concentrated with parameter $q >0$  and observable diameters $\sigma_n >0$,  if there is a constant $C>0$ such that for any sequence of $1$-Lipschitz maps $f_n:E_n\to \C$, for any $n \in \N$ and $t \geq 0$:
\begin{align*}
    \Pp\pt{\verti{f_n(X_n) - \E\br{f_n(X_n)}}\geq t} \leq C e^{- \frac 1 C \pt{\frac {t}{ \sigma_n}}^q}.
\end{align*}
We denote $X_n \propto \cal E_q(\sigma_n)$ this notion of concentration.
\end{defi}
As explained in the beginning of this paper, we will omit most of the $n$ indices for a better readability. We also use the following generalization of the definition: $X \propto \cal E_q(\sigma_n) + \cal E_{q'}(\sigma_n')$ means that the above property holds with the right hand term side of the inequality replaced by $ C e^{- \frac 1 C \pt{\frac {t}{ \sigma_n}}^q} +  C e^{- \frac 1 C \pt{\frac {t}{ \sigma'_n}}^{q'}}$.

For practical purposes in this paper, we only consider $\R^n$ endowed with the Euclidean norm, or equivalently $\R^{p \times n}$ endowed with the Frobenius norm. In this setting, any standard Gaussian vector is concentrated independently of its dimension: if the entries of $X$ are i.i.d. sampled from a standard Gaussian random variable, then $ X \propto \cal E_2(1)$ (\cite{ledoux2001concentration}). The class of Lipschitz concentrated random vectors goes way beyond Gaussian vectors, see \cite{ledoux2001concentration}, and \cite{talagrand1995concentration} and \cite{TaoBook} for more examples.

We deduce immediately from the definition the following property: 
\begin{prop}\label{LipschitzTransformsConcentration} If $X_n \propto \cal E_q(\sigma_n)$ and $\varphi_n$ is a sequence of  $k_{\varphi_n}$-Lipschitz functions on the support of $X_n$, then $\varphi_n(X_n) \propto \cal E_q(k_{\varphi_n}\sigma_n)$ with new concentration constants that do not depend on $\varphi_n$.\end{prop}

Let us mention a simple link between observable diameters and variance for random variables.

\begin{prop}\label{VarianceConcentration} If a random variable $\xi$ is $\cal E_q(\sigma_n )$ concentrated, then \newline
$\verti{\xi - \E[\xi]}~\leq~O\pt{\sigma_n \pt{ \log n}^{\frac 1 q}}$ a.s. and $\Var[\xi ] \leq O(\sigma_n^2)$. These bounds only depend on the concentration constant of $\xi$, not on the whole law of $\xi$.
\end{prop}

\begin{proof} For the first assertion, let us choose $t_n = \sigma_n \pt{2 C \log n}^{\frac 1 q}$ in the concentration inequality. We obtain:
\begin{align*}
    \Pp\pt{\verti{\xi - \E[\xi]} \geq t_n} \leq C e^{- \frac 1 C \pt{\frac {t_n}{\sigma_n}}^q} = \frac C {n^2}.
\end{align*}
This bound is summable, so Borel Cantelli lemma implies that $\verti{\xi - \E[\xi]} \leq t_n$ a.s.

For the second assertion, we integrate the concentration inequality:
\begin{align*}
  \Var[\xi ]  &= \int_0^\infty  \Pp\left(\verti{\xi - \E[\xi]}^2 \geq t \right) dt \\
    &\leq C \int_0^\infty   \rm{exp}\pt{- \frac 1 C \frac {t^{\frac q 2}}{ \sigma_n^q}} dt \\
    &= \frac{2 C^{1 + \frac 2 q} }{q} \Gamma\pt{\frac 2 q } \sigma_n^2
\end{align*} \end{proof}

 Thanks to a classic argument of $\epsilon$-nets (\cite{TaoBook}), we also  have a control on the spectral norm of concentrated random matrices:

\begin{prop}\label{ConcentrationSpectralNorm}
If a random matrix $X_n \in \R^{p_n \times n}$ is $ \propto \cal E_2(\sigma_n)$ concentrated, then $\vertiii{X_n - \E[X_n]} \leq O((n+p_n)^\half \sigma_n)$ a.s.
\end{prop}
\begin{proof} Upon centering and rescaling, we can assume without loss of generality that $X$ is centered and $\propto \cal E_2(1)$ concentrated.

 Let us consider a covering of the unit Euclidean ball of $\R^{p}$ with $9^p$ balls of radius $ 1/3$, centered in $u_1, \dots, u_{9^p}$ vectors belonging to the unit Euclidean ball. Similarly let us consider a covering of the unit Euclidean ball of $\R^{n}$ with $9^n$ balls of radius $1/3$, centered in $v_1, \dots, v_{9^n}$ vectors belonging to the unit Euclidean ball.
The applications $M \mapsto u_i^\top M v_j$ all are $1$-Lipschitz with respect with the Frobenius norm. From Propositions \ref{LipschitzTransformsConcentration} and \ref{VarianceConcentration} there is  $C>0$ such that for any $i$, $j$ and $t\geq0$:
\[ \Pp\pt{\verti{ u_i^\top X v_j}\geq t} \leq C e^{-\frac {t^2} C }.\]
 A union bound over the pairs $(i,j)$ shows that with $t_n =  C^\half \pt{(n+p)\log 9 + 2 \log n}^\half  $:\[ \Pp\pt{\sup_{i,j} |u_i^\top X v_j| \geq t_n} \leq C 9^{n+p} e^{-\frac {t_n^2} C } =  \frac C {n^2} .\]
This bound is summable, so by Borel-Cantelli lemma, $\sup_{i,j} |u_i^\top X v_j| \leq t_n$ for $n$ large enough a.s.

For any $u \in \R^p$ and $v \in \R^n$ unit-normed vectors, we can find $i$ and $j$ such that $\|u-u_i\|$ and $\|v - v_j\| \leq \frac 1 3$, hence:
\begin{align*}
 | u^\top X v | &\leq \left| u_i^\top X v_j \right| +\left| (u-u_i)^\top X v_j \right| + \left|u^\top X (v-v_j) \right| \\
 &\leq |u_i^\top X v_j |+ \|u-u_i\|\vertiii{ X} \| v_j\| + \|u^\top\| \vertiii{X} \|v-v_j\| \\
 &\leq \sup_{i,j} |u_i^\top X v_j| + \frac 2 3 \vertiii{X} .
\end{align*}
We deduce that $\vertiii{X} = \sup_{\|u\|=\|v\|=1} | u^\top X v | 
 \leq \sup_{i,j} |u_i^\top X v_j| + \frac 2 3 \vertiii{X}$, thus $\vertiii{X} \leq 3 \sup_{i,j} |u_i^\top X v_j| \leq 3 t_n \leq O((n+p)^\half)$ a.s.
    
\end{proof}

The next proposition, sometimes referred to as Hanson-Wright type inequality after \cite{hanson1971bound}, establishes the concentration of quadratic forms and will prove itself to be crucial in the analysis of resolvents. A proof of this result can be found in \cite{adamczak2015}, Theorem 2.4.

\begin{prop}[Hanson-Wright] \label{HansonWright} Let $x \in \R^n$ be a random vector, such that $x\propto \cal E_2(1)$ and $\vertii{\E[x]}$ is bounded. Then uniformly in any deterministic matrix $A \in \C^{n \times n}$, we have $x^\top A x \propto \cal E_2(\vertif{A}) + \cal E_1(\vertiii{A}) $, in the sense that the new concentration constants can be taken uniformly in $A$.
\end{prop}

We end this section with a simple lemma to bound the variance of a product of random variables when one of them is a.s. bounded. 

\begin{prop}[Variance of a product] \label{VarianceProduct} Let $s$ and $t$ be $L^2$ complex random variables with $\verti{s}$ a.s. bounded by $\vertii{s}_\infty > 0$. Then: 
\[ \Var[s t]  \leq 2 \vertii{s}_\infty^2  \Var[t]  + 2\verti{ \E[t]}^2 \Var[s] \]
\end{prop}

\begin{proof} We decompose the product into $s t = s (t - \E[t]) + s \E[t]$ and use the inequality $\Var[a + b] \leq 2 \Var[a] + 2 \Var[b]$:
\begin{align*}
    \Var[s t] &= \Var\br{s (t - \E[t]) + s \E[t]} \\
    &\leq 2  \Var\br{s (t - \E[t])} + 2  \Var\br{ s \E[t]} \\
    &\leq 2 \E\br{|s|^2 \verti{t - \E[t]}^2} + 2\verti{ \E[t]}^2 \Var[s] \\
    &\leq 2 \vertii{s}_\infty^2  \Var[t]  + 2\verti{ \E[t]}^2 \Var[s].
\end{align*}
\end{proof}

\newpage
\section{Properties of resolvents matrices}

This section deals with the resolvent of sample covariance matrices, and introduces two important companion objects: the Leave One Out (LOO) resolvent and the co-resolvent.

Given a matrix $M \in \R^{p \times p}$ and a spectral parameter $z \in \C$ that is not an eigenvalue of $M$, the resolvent of $M$ is the matrix $\cal G_M(z) = (M- z I_p)^{-1}$. We will often discard the $z$ in the notations for better readability.

By resolvent identity we refer to the relation between invertible matrices $A-B = A(B^{-1}-A^{-1})B$,  and its immediate corollary for resolvents $\cal G_M - \cal G_N = \cal G_M (N-M) \cal G_N $.

\subsection{Sample covariance resolvent} 

If $X \in \R^{p_n \times n}$ is a rectangular matrix, we set its sample covariance $K := \frac 1 n X X^\top $. Given a spectral parameter $z \in \R^{*-}$ or $\C^+$, we define the resolvent matrix $\cal G := \cal G_K(z)$. In other terms:
\[ \cal G =  \pt{K- z I_p}^{-1} = \pt{ \frac 1 { n} X X^\top - z I_p}^{-1} \] 

To unify the notations we let $\eta = |z|^{-1}$ when $z \in \R^{*-}$ and $\eta = { \Im(z)}^{-1}$ when $z \in \C^+$. Note that we always have $\eta |z| \geq 1$.

Let us establish some elementary properties of the sample covariance resolvent. 

\begin{prop}\label{BoundG}\label{BoundGX} $\cal G$ is well defined, $\vertiii{\cal G} \leq \eta$ and  $  \vertiii{\cal G X}
    \leq n^\half \pt{  2^\half \eta |z|^\half }$.\end{prop}
    
\begin{proof} The eigenvalues of $K$ are real non negative, thus all eigenvalues of $K-zI_p$ are bounded from below by $d(z,\R^+)$, and $K-zI_p$ is invertible with $\vertiii{\pt{K-zI_p}^{-1}}=\vertiii{\cal G} \leq \eta$. From the identity $\cal G  K   = I_p + z \cal G $ we deduce that $ \vertiii{\cal G X}^2 = \vertiii{\pt{\cal G X X^\top }\cal G^\dg} \leq\vertiii{\cal G n K}  \cdot \vertiii{\cal G} \leq n \pt{1+ \eta {|z|}}\eta \leq  {2n\eta^2|z|} $. \end{proof}

In the following proposition $\Im(M)$ denotes the imaginary part of a complex matrix $M$.

\begin{prop}\label{ImG} 
 $\Im\pt{\cal G} =  \Im(z) \cal G \cal G^\dg$ and $\Im\pt{z \cal G} =  \Im(z) K \cal G \cal G^\dg$. In particular $\Im\pt{\cal G}$ and  $\Im\pt{z \cal G}$ are positive semi-definite.
\end{prop}

\begin{proof} Using the resolvent identity, $2i \Im(\cal G) = \cal G - \ov{\cal G} =  \cal G (- \ov z I_p + z I_p ) \ov{\cal G} = 2i \Im(z) \cal G \cal G^\top$. Given that $z \cal G = K \cal G - I_p$ we also get $ \Im(z \cal G)  = K \Im(\cal G) = \Im(z) K \cal G \cal G^\top$. The matrices $\cal G \cal G^\top$ and $K$ commute and are both positive semi-definite, which achieves the proof.
\end{proof}

The resolvent map $M \mapsto (M-zI_p)^{-1}$ is in general not Lipschitz with respect to the Frobenius norm when $\vertiii{M}$ is not bounded. In the specific case of the sample covariance resolvent however, we can always retrieve a Lipschitz property that will greatly simplify all concentration analysis.

\begin{prop} \label{LipG} The map $X \mapsto \cal G(X) = \pt{  \frac 1 n  X X^\top - z I_p}^{-1}$ is Lipschitz with respect to the Frobenius norm with parameter $n^{-\half} \pt{ {2^{\frac 3 2} \eta^2 |z|^\half}{ } }$.
\end{prop}

\begin{proof} Let us consider $X$ and $H \in \R^{p \times n}$. We have:
\begin{align*}
    \cal G(X) - \cal G(X+H) &=  \cal G(X) \frac 1 n \pt{(X+H)(X+H)^\top - X X^\top} \cal G(X+H) \\
    &= \frac 1 n  \cal G(X)  \pt{XH^\top +H(X+H)^\top } \cal G(X+H) .
\end{align*}
We bound this expression in the following way: 
\begin{align*}
   & \vertif{ \cal G(X) - \cal G(X+H)} \\
    &\leq  
 \frac 1 n\pt{  \vertiii{\cal G(X)X} \vertif{H^\top} \vertiii{ \cal G(X+H)} + \vertiii{\cal G(X)} \vertif{H} \vertiii{(X+H)^\top \cal G(X+H)}} \\
    &\leq \frac 1 n \pt{n^\half \pt{  2^\half \eta |z|^\half } \vertif{H} \eta + \eta\vertif{H}n^\half \pt{  2^\half \eta |z|^\half }} = n^{-\half} \pt{ {2^{\frac 3 2} \eta^2 |z|^\half}{ } }  \vertif{H}.
\end{align*}
\end{proof}

As an immediate consequence, we obtain by Proposition \ref{LipschitzTransformsConcentration} the following key result for the concentration of sample covariance resolvents:

\begin{cor} \label{ConcentrationG} If $X \propto \cal E_2(1 )$, then $\cal G \propto \cal E_2  \pt{ n^{-\half} \eta^2  |z|^\half{  }} $.
\end{cor}

\subsection{LOO resolvent and co-resolvent}

Let $x_1$, $x_2, \dots, x_n \in \R^p$ be the columns of $X$. For a better readability we denote the first column $x_1$ simply $x$.

We define the Leave One Out (LOO) matrix $X_- = \pt{\bf 0_p, x_2, \dots , x_n}$, the  LOO sample covariance $K_- = \frac 1 n X_- {X_-}^\top$ and the LOO resolvent $\cal G_- = \cal G_{K_-}(z)$. We have the following identities :
\begin{align*}
    K_- &= \frac 1 n \sum_{j \geq 2} x_j x_j^\top = K - \frac 1 n x x^\top,\\
 \cal G_- &=  \pt{K_- - z I_p}^{-1} = \pt{ \frac 1 { n} X_- {X_-}^\top - z I_p}^{-1}.
\end{align*} 
We also define the co-sample covariance matrix by swapping $X$ and $X^\top$: $\check K = \frac 1 n X^\top X$, and the co-resolvent: $\check{\cal G }= \cal G_{\check K}(z)$. In other terms:
\[ \check{\cal G } = \pt{\check K- z I_n}^{-1} = \pt{ \frac 1 { n}  X^\top X - z I_n}^{-1}. \] 
The LOO resolvent and the  co-resolvent naturally inherits all bounds and concentration properties of the sample covariance resolvent: $\vertiii{\cal G_-}$ and $\vertiii{\check{\cal G}} \leq \eta$, and if $X \propto \cal E_2(1)$, then $\cal G_-$ and $\check{\cal G} \propto \cal E_2\pt{ n^{-\half} \eta^2  |z|^\half{  }}  $.
When $X$ has an additional structure of independence between the columns, the first column $x$ is independent from the LOO resolvent $\cal G_-$, leading to further interesting properties.

\begin{prop}\label{VarianceQuadraticForms} Assume that $X \in \R^{p \times n}$ has i.i.d. columns, that  $X$ is $\propto~\cal E_2(1)$ concentrated, and that  $\vertii{\E[x]}$ and $\vertiii{\E\br{x x^\top}}$ are bounded. Let $\sigma^2 = \eta^2\pt{1 + n^{-1} \eta^2 |z|}$.
Then uniformly in any deterministic matrix $B \in \C^{p \times p}$ with $\vertif{B} \leq 1$:
\begin{align}        \label{VarianceQuadraticForms3}  \Var\br{x^\top \cal G_- x}  &\leq O\pt{p \sigma^2 }\\
\label{VarianceQuadraticForms1}
    \Var\br{x^\top B \cal G_- x}  &\leq O\pt{ \sigma^2 } \\
      \label{VarianceQuadraticForms2}  \Var\br{x^\top \cal G_- B \cal G_- x}  &\leq O\pt{\eta^2 \sigma^2 } 
\end{align}
\end{prop}

\begin{rem}
These variance bounds correspond exactly to the ones we would obtain by considering that $\cal G_-$ is deterministic and applying the concentration of quadratic forms given by Proposition \ref{HansonWright}. The full concentration property of these  quantities remains however unclear.
\end{rem}

\begin{proof} We denote $\Sigma = \E\br{x x^\top}$. Remark that $\Sigma = \frac 1 n \E\br{ \sum_{i=1}^n  x_i x_i^\top} = \frac 1 n \E\br{ X X^\top} =  \E[K]$. Without loss of generality we can assume that $\vertii{\E[x]}\leq 1$ and  $\vertiii{\Sigma} \leq 1$. 
Let us first prove in full details the second inequality.

For any deterministic matrix $M \in \R^{p \times p}$ with $\vertif{M} \leq \eta $, from Proposition \ref{HansonWright} we have $x^\top M x \propto \cal E_2\pt{\eta } + \cal E_1\pt{\eta}$, thus $\Var\br{x^\top M x}  \leq O\pt{\eta^2 } $. If $B$ is deterministic with $\vertif{B} \leq 1$, then $\vertif{B\cal G_-} \leq \eta$ and using independence:
\begin{align*}
    \E\br{\verti{x^\top B \cal G_- x - \Tr(B \cal G_- \Sigma)}^2 } &= \E\br{\E\br{\left.\verti{x^\top B \cal G_- x - \Tr(B \cal G_- \Sigma)}^2 \,\right|\,  \cal G_- \, }} \\
   &= \E_{\cal G_-} \br{\E_x\br{\verti{x^\top M x - \Tr(M \Sigma)}^2}(M = B \cal G_- )} \\
    &= \E_{\cal G_-} \br{\Var_x\br{x^\top M x }(M = B \cal G_- )} \leq O\pt{\eta^2 }.
\end{align*}

On the other hand  $M \mapsto \Tr(B M \Sigma)$ is $1$-Lipschitz with respect to the Frobenius norm, therefore by Proposition \ref{LipschitzTransformsConcentration} $\Tr\pt{B \cal G_- \Sigma}\propto \cal E_2\pt{ n^{-\half}  \eta^2  |z|^\half{  }} $, and by Proposition \ref{VarianceConcentration} $\Var[\Tr(B \cal G_- \Sigma)] \leq O\pt{ n^{-1}  \eta^4 |z|}$. We combine both estimates to obtain the inequality (\ref{VarianceQuadraticForms1}):
\begin{align*}
    \Var\br{x^\top B \cal G_- x} 
    &= \E\br{\verti{x^\top B \cal G_- x - \Tr(B \E[ \cal G_-] \Sigma)}^2} \\
    &\leq 2 \E\br{\verti{x^\top B \cal G_- x - \Tr\pt{ B \cal G_- \Sigma}}^2} + 2 \E\br{\verti{\Tr\pt{ B \cal G_- \Sigma} -  \Tr( B\E[ \cal G_-] \Sigma)}^2} \\
    &\leq O\pt{\eta^2 } + O\pt{ n^{-1}  \eta^4 |z|} = O(\sigma^2).
\end{align*}
At no point did the $O(\cdot)$ notations depend on $B$ in the estimates, as explained in the concentration section.

The first inequality (\ref{VarianceQuadraticForms3}) can be seen as a consequence of the  second inequality applied to $B = p^{-\half} I_p$ which satisfies $\vertif{B} \leq 1$.
The proof of the third inequality (\ref{VarianceQuadraticForms2}) is quite similar to the first one. For this reason we will point out only the key differences. Using the deterministic bound $\vertif{\cal G_- B \cal G_-} \leq \eta^2$ we get:
\[\E\br{\verti{x^\top \cal G_- B \cal G_- x - \Tr(\cal G_- B \cal G_- \Sigma)}^2} \leq O\pt{\eta^4 }. \]
On the other hand the map $M \mapsto \Tr(M B M \Sigma) = \Tr( B M \Sigma M)$ is $ {\eta  {} }$- Lipschitz  on the set $\left \{ M \in \C^{p\times p} \rm{~with~} \vertiii{M} \leq \eta \right\}$. Indeed:
\begin{align*}
 \vertif{ M \Sigma M -  M' \Sigma M' } &\leq
  \vertiii{M} \vertiii{\Sigma}\vertif{M- M'} +\vertif{M-M'} \vertiii{\Sigma}\vertiii{M'}\\
  &\leq \eta \vertif{M- M'}.
\end{align*}
This implies that $\Tr(\cal G_- B \cal G_- \Sigma) \propto  \cal E_2\pt{ n^{-\half}  \eta^3  |z|^\half{  }} $ and $\Var[\Tr(\cal G_- B \cal G_- \Sigma)]  \leq O\pt{ n^{-1}  \eta^6 |z|}$. Combining the previous two estimates proves the inequality (\ref{VarianceQuadraticForms2}). 
\end{proof}

We wrap up this section by mentioning some of the links that exist between the resolvent, the LOO resolvent and the co-resolvent. We need the following linear algebra results:

\begin{lem}[Sherman-Morrison] \label{ShermanMorrison} Given an invertible matrix $M \in \C^{ p \times p} $ and two vectors $u,v \in \C^p$, $M + u v^\top$ is invertible if and only if $1 + v^\top M^{-1} u \neq 0 $, in which case the following identities hold true:
\begin{align*}
    \pt{M + u v^\top}^{-1} &= M^{-1} - \frac{M^{-1} u v^\top M^{-1}}{1 + v^\top M^{-1} u }, \\
    \pt{M + u v^\top}^{-1} u  &= \frac{M^{-1} u }{1 + v^\top M^{-1} u }.
\end{align*}
\end{lem}
\begin{lem}
 $\frac 1 n X^\top \cal G X = I_n + z \check{\cal G}$.
\end{lem}
\begin{proof}
Since $X^\top \left( \frac 1 n X X^\top - z I_p\right)  =  \left( \frac 1 n X^\top X  - z I_n\right) X^\top $, we have $ \frac 1 n X^\top \cal G X =    \frac 1 n \left( \frac 1 n X^\top X  - z I_n\right) ^{-1} X^\top  X =   I_n + z \check{\cal G}$.
\end{proof}

As an immediate application of these lemmas, we get the following identities, which we will refer to as LOO identities. These will be extremely useful to exploit any column independence structure in a random matrix, and also to prove bounds and concentration properties that  otherwise would have been far from obvious.

\begin{prop}[LOO identities]\label{LOOId} $1 + \frac 1 { n} x^\top \cal G_- x \neq 0$, and:
\begin{align*}\frac 1 n x^\top \cal G x    
 &= 1 - \frac 1 {1 + \frac 1 { n} x^\top \cal G_- x} =    1 + z \check{\cal G}_{11}, \\
    \cal G &= \cal G_- - \frac 1 {n} \frac{\cal G_- x x^\top \cal G_-}{1 + \frac 1 { n} x^\top \cal G_- x}  = \cal G_- + \frac z n \check{\cal G}_{11} \cal G_- x x^\top \cal G_-,\\
    \cal G x &=  \frac{\cal G_- x}{1 + \frac 1 {n} x^\top \cal G_- x} = - z \check{\cal G}_{11} \cal G_- x. \end{align*}
\end{prop}

\newpage

\section{First Deterministic Equivalent}

\subsection{General properties of the deterministic equivalents}

We place ourselves under the Assumptions \ref{Ass} for the matrix $X$, and we set the sample covariance matrix $K = n^{-1} X X^\top$. For a better readability we will note the first column of $X$ as $ x_1 = x$. The other columns are i.i.d. sampled with the same law as $x$, hence $\Sigma = \E[K] = \E[x x^\top] $. Without loss of generality, we can assume that $\vertii{\E[x]} \leq 1$ and $\vertiii{\Sigma} \leq 1$.

We introduce respectively the classical resolvent $\cal G = \pt{\frac {X X^\top} n  - z I_p}^{-1}$, the Leave-One-Out (LOO) resolvent $\cal G_- =  \pt{\frac {{X X^\top - x x^\top}} n  - z I_p}^{-1}$  and the co-resolvent $\check{\cal G} =  \pt{\frac {X^\top X} n   - z I_n}^{-1}$ associated to $X$. An in-depth analysis of these resolvent matrices can be found in Section 4.

The general form of our deterministic equivalents is given by: \[\bf G^{\frak l}=  \pt{\frac z {\frak l} \Sigma - z I_p}^{-1} \] for some sequence of deterministic parameters $\frak l \in \C$ that may vary with $z$. 

 We will first give a sufficient condition on $\frak l$ so that $\bf G^{\frak l}$ is well defined and bounded. To this end we define $\Omega=(- \infty, z]$ if $z \in \R^{*-}$, and $\Omega = \{ \frak l \in \C$ such that $\Im(\frak l) \geq \Im(z)$ and  $\Im(z^{-1} {\frak l})\geq 0\}$ if $z \in \C^+$. We also recall the definition of $\eta = |z|^{-1}$ if $z \in \R^{*-}$ and $\eta = { \Im(z)}^{-1}$ if $z \in \C^+$.

\begin{lem}\label{AutomaticBoundDetEq} If $\frak l \in  \Omega$, then $\bf G^{ \frak l}$  is well defined, and $\vertiii{\bf G^{ \frak l}}$ and $\vertiii{\frac z {\frak l} \bf G^{ \frak l} } \leq \eta$.\end{lem}

\begin{proof} The eigenvalues of $\frac z {\frak l} \Sigma - z I_p$ are $\frac z {\frak l} \lambda - z $ for $\lambda \in \Sp \Sigma \subset \R^+$.
If $z \in \R^{*-}$, then $\frac z {\frak l} \geq 0$ and $\frac z {\frak l} \lambda - z \geq |z|$. If $z \in \C^+$, then $\Im\pt{ \frac z {\frak l} } \leq 0$ and ${\Im\pt{\frac z {\frak l} \lambda - z }} \leq - \Im(z)$.
In both cases, all the eigenvalues of $\frac z {\frak l} \Sigma - z I_p$ are greater or equal than $\eta^{-1}$  in modulus, hence $\bf G^{ \frak l} = \pt{\frac z {\frak l} \Sigma - z I_p}^{-1}$  is well defined and $\vertiii{\bf G^{ \frak l}} \leq \eta$.

The argument for $\frac z {\frak l} \bf G^{\frak l}$ is similar: the eigenvalues of $ \Sigma - \frak l I_p$ are $ \lambda - \frak l $ for $\lambda \in \Sp \Sigma $.
If $z \in \R^{*-}$ then $\frak l \leq z$ and $\lambda - \frak l \geq |z|$. If $z \in \C^+$ then $\Im(\frak l) \geq \Im(z)$ and ${\Im\pt{\lambda - \frak l}}\leq - \Im(z)$.
In both cases, all the eigenvalues of $ \Sigma - \frak l I_p$ are greater or equal than $\eta^{-1}$  in modulus, hence  $\vertiii{\frac z {\frak l} \bf G^{ \frak l} } =\vertiii{ \pt{\Sigma - \frak l I_p}^{-1}} \leq \eta$.
\end{proof}

\subsection{Introduction of parameters $\frak a$ and $\frak b$}  Let us consider $\frak a =  z + \frac z n x^\top \cal G_- x$ and $\frak b = \E[\frak a] = z + \frac z n \Tr\pt{\Sigma \E\br{\cal G_-}}$. In the rest of this section we  will prove that $\bf G^{\frak b}$ is close to $\E\br{\cal G}$.  The precise statement constitutes Theorem \ref{FirstDetEq}.

We can rewrite the LOO identities of Proposition \ref{LOOId} in the following fashion:  \begin{align*} \frac 1 n  x^\top \cal G x &= 1 - \frac {z }{\frak a} = 1 + z{\check{\cal G}_{11}},  \\ 
    \cal G &= \cal G_- - \frac z {\frak a}\frac 1 {n} {\cal G_- x x^\top \cal G_-}, \\ \label{aId3}
    \cal G x &=   \frac z {\frak a} \cal G_- x .
\end{align*} 
In particular $\frak a= -  {\check{\cal G}}_{11}^{-1}$. We also deduce the following properties:
\begin{prop}\label{PropertiesAandB}$\frak a$ and $\frak b \in \Omega$. $ \bf G^{ \frak b}$ is well defined, and $\verti{ {\frak a}^{-1}}$ and $\vertiii{\frac z {\frak b}  \bf G^{ \frak b}}$ are bounded from above by  $\eta$.
\end{prop}

\begin{proof} If $z \in \R^{*-}$, $\cal G_-$ is a real positive semi-definite matrix and $\frac {\frak a} z = 1 + \frac 1 n x^\top \cal G_- x \geq 1$, hence $\frak a \leq z$. If $z \in \C^+$,  $\Im\pt{\cal G_-}$ and $\Im\pt{z \cal G_-}$ are positive semi-definite matrices using Proposition \ref{ImG}, thus $\Im(\frak a) = \Im(z) + \frac 1 n x^\top \Im\pt{z \cal G_-} x \geq \Im(z)$, and $\Im\pt{\frac {\frak a} z} = \frac 1 n x^\top \Im\pt{\cal G_-} x \geq 0$. In any case,  $\frak a \in \Omega$, and $\Omega$ is stable through expectation so $\frak b \in  \Omega$ as well. The last assertions are immediate consequences of Lemma \ref{AutomaticBoundDetEq} and the definition of $\Omega$.
\end{proof}

\subsection{Proof of the first deterministic equivalent } Let us recall the definition of $\eta = |z|^{-1}$ if $z \in \R^{*-}$ and $\eta = { \Im(z)}^{-1}$ if $z \in \C^+$. We can now prove the following result:

\begin{thm}[First equivalent] \label{FirstDetEq} Let $z(n)$ be a sequence of spectral arguments, either in $\R^{*-}$ or in $\C^+$. If $\eta^{10} |z|^3 \leq O(n)$ and if $|z|$ or $\Im(z)$ is bounded, then:  \[ \vertif{\E[\cal G] - \bf G^{{\frak b}} } \leq O\pt{ \frac {\eta^5 |z|^{\frac 3 2}}{n^\half} } \] 
\end{thm}

\begin{rem} In the case where $z$ stays bounded away from $\R^+$, the theorem boils down to a $O(n^{-\half})$ estimate like as implied by \cite{LC21}, Theorem 4.

The assumption $\eta^{10} |z|^3 = O(n)$ is not strictly necessary to obtain an explicit bound. Indeed we can always obtain a $O\pt{n^{-\half} \kappa }$ bound where $\sigma^2 = \eta^2\pt{1+n^{-1} {\eta^2|z|} }$ and  $\kappa ={\sigma^2 \eta^2 |z| +  \sigma \eta^4 |z|^{\frac 3 2} }$.

However $\kappa \geq  \eta^5 |z|^{\frac 3 2}$ so the above bound will not converge to $0$ without the additional assumption that $\eta^{10} |z|^3 = O(n)$. The other assumption can be rewritten as  $1 \leq O(\eta)$, which simplifies the results and is not restrictive for practical applications. Under these two hypothesis, $\kappa = O( \eta^5 |z|^{\frac 3 2})$ which gives the same estimate.
\end{rem}

The proof of Theorem \ref{FirstDetEq} can be organized in three steps:
\subsubsection{First step:}We introduce $\cal G_-$ using the LOO Identities:
 \begin{align*}
   \cal G\left( {\frac {z}{\frak b}} \Sigma -  x x^\top  \right) &=  \cal G {\frac {z}{\frak b}} \Sigma - {\frac {z}{{\frak a}}}  \cal G_- x x^\top  \\
     &=  {\frac {z}{\frak b}} \pt{  \pt{\cal G-\cal G_-} \Sigma   +   \cal G_- \pt{\Sigma - x x^\top}  + \pt{ {1  - {\frac {\frak b}{\frak a}}}} \cal G_- x x^\top }.\end{align*}
  $X X^\top$ is the sum of $n$ $\iid$ copies of $x x^\top$, and $ \E\br{\cal G_- \pt{\Sigma - x x^\top}} = 0 $ by independence, leading to: 
\begin{align*} \label{DecompositionFirstEq}
     \E\br{\cal G} - \bf G^{\frak b} &= 
     \E\left[\cal G\left( {\frac {z}{\frak b}} \Sigma - \frac 1 n X X^\top \right)  \bf G^{\frak b}\right] \nonumber\\
     &=      \E\left[\cal G\left( {\frac {z}{\frak b}} \Sigma - \frac 1 n \sum_{i=1}^n x_i x_i^\top \right)  \bf G^{\frak b}\right]\\
     &= \E\left[\cal G\left( {\frac {z}{\frak b}} \Sigma -  x x^\top  \right) \right] \bf G^{\frak b} \nonumber\\
     &={\frac {z}{\frak b}} \pt{  \E\left[\cal G - \cal G_-\right] \Sigma  +  \E\left[ \frac{\frak a-\frak b}{\frak a}\cal G_- x x^\top \right] } \bf G^{\frak b}.
\end{align*}

\subsubsection{Second step:} We estimate the above quantities using concentration and the properties of LOO quadratic forms. 

First note that the term $\sigma$ appearing in Proposition \ref{VarianceQuadraticForms} boils down to $\sigma = O(\eta)$ under our Assumptions. Indeed $\eta^2 |z| \leq \eta^{\frac 5 2} |z|^{\frac 3 2} \leq O(\eta ^5 |z|^{\frac 3 2} ) \leq O(n^\half) $, so $\sigma^2 = \eta^2\pt{1+n^{-1} {\eta^2|z|} } = O(\eta^2)$. 

We have estimates on the variance of both $\frak a$ and $\frak a^{-1}$: indeed $\Var[\frak a]  = {|z|^2}{n^{-2}}\Var\br{x^\top \cal G_- x}\leq O( n^{-1} \eta^2 |z|^2 )$ from Proposition \ref{VarianceQuadraticForms}, and $ {\frak a}^{-1} = \check{\cal G}_{11} \propto \cal E_2  \pt{ n^{-\half} \eta^2  |z|^\half{  }} $ so $\Var\br{ {\frak a}^{-1}} \leq O\pt{n^{-1} \eta^4 |z|}$.

 Now let us consider a matrix $B \in \C^{ p \times p}$ with $\vertif{B} \leq 1$. We have $ \Tr\left( B  \E[\cal G - \cal G_-] \right) =  \E\br{ \frac z {\frak a}  \frac 1 n \Tr\pt{B \cal G_- x x^\top \cal G_-}} = \frac  z n \E\br{\frac \zeta {\frak a} } $ with  $\zeta = x^\top \cal G_- B \cal G_- x$.

First note that $\verti{\E[\zeta]}=\verti{\Tr\pt{\E\br{ \cal G_- B \cal G_- } \Sigma}} \leq \eta^2 p^\half  \leq O ( n^\half \eta^2)$, and from Proposition \ref{VarianceQuadraticForms} $\Var\br{\zeta} \leq O\pt{\eta^4}$ uniformly in $B$. Using the deterministic bound $\verti{\frak a^{-1}} \leq \eta$, we have:
\begin{align*}   
 \verti{\E\br{\frak a^{-1}  \zeta  } } &\leq  \verti{\E\br{\frak a^{-1} \pt{\zeta -  \E[\zeta]}  } } + \verti{ \E[\frak a^{-1}]\E[\zeta] } \\
 &\leq \eta \Var[\zeta]^\half + \eta \verti{\E[\zeta]} \leq O\pt{ n^\half \eta^3}
\end{align*}We deduce that:
 $\vertif{\E[\cal G - \cal G_-]} = \sup_{\vertif{B}=1}\verti{ \Tr\left( B  \E[\cal G - \cal G_-] \right) } \leq O\pt{n^{- \frac 1 2} 
|z| \eta^3  }$.

For the second term, $\Tr\pt{\E\br { B  \frac{\frak a-\frak b}{\frak a}\cal G_- x x^\top}} = \E\br{(\frak a - \E[\frak a]) \frac \xi {\frak a}}$ with $\xi = x^\top B \cal G_- x$. From Proposition \ref{VarianceQuadraticForms} $\Var[\xi] \leq O\pt{\eta^2}$ uniformly in $B$, and $\verti{\E[\xi]}=\E\br{\Tr\pt{B \cal G_- \Sigma}} \leq p^\half \eta \leq O ( n^\half \eta)$. We bound the variance of $\frac \xi {\frak a}$ using Propositions \ref{PropertiesAandB} and \ref{VarianceProduct}: 
\begin{align*}
\Var\br{\frac \xi {\frak a}}  &\leq 2 \, \vertii{{\frak a}^{-1} }_\infty^2 \Var[\xi]  + 2 \Var\br{{\frak a}^{-1} } \verti{ \E[\xi]}^2 \\
&\leq O\pt{\eta^4+ \eta^6 |z|} \leq O(\eta^6 |z|)
\end{align*}
where we used that $1 \leq O(\eta) \leq O(\eta^2 |z|)$.
Cauchy-Schwartz inequality again implies that:
\[ \verti{\E\br{(\frak a - \E[\frak a]) \frac \xi {\frak a}}}^2 \leq \Var\br{\frak a} \Var\br{\frac \xi {\frak a}}   \leq O\pt{n^{-1} \eta^8 |z|^3 }\] We deduce that:
\begin{align*}
    \vertif{ \E\left[ \frac{\frak a-\frak b}{\frak a}\cal G_- x x^\top \right]} = \sup_{\vertif{B}=1}\verti{\Tr\pt{\E\br { B  \frac{\frak a-\frak b}{\frak a}\cal G_- x x^\top}} }
    \leq O\pt{n^{-\half}\eta^4 |z|^{\frac 3 2}}
\end{align*}
\subsubsection{Third step:} We wrap up the proof by combining the previous estimates:
\begin{align*}
\vertif{\E\br{\cal G} - \bf G^{\frak b}} &\leq  \vertiii{{\frac {z}{\frak b}}  \bf G^{\frak b}}\cdot \pt{  \vertif{\E\left[\cal G - \cal G_-\right] }\vertiii{\Sigma } + \vertif{ \E\left[ \frac{\frak a-\frak b}{\frak a}\cal G_- x x^\top \right] }} \\
&\leq \eta  \cdot \pt{ {O\pt{n^{- \frac 1 2} \eta^3  |z|} + O\pt{n^{-\half} \eta^4 |z|^{\frac 3 2}} }  } \\
&\leq  n^{-\half} O\pt{\eta^5 |z|^{\frac 3 2} }.
\end{align*}
\newpage

\section{Second deterministic equivalent }

We continue our analysis under the same assumptions as in Theorem \ref{FirstDetEq}.

\subsection{Reformulation as a fixed point problem}

Let us recall the definitions of $\nu = \mmp\pt{\gamma_n } \boxtimes \mu_{\Sigma}$, $\check \nu = (1-\gamma_n ) \delta_0 + \gamma_n  \nu$ and $\bf G^{\frak l}=  \pt{\frac z {\frak l} \Sigma - z I_p}^{-1} $. We let ${\frak c} = -  { g_{\check \nu}}^{-1}$. In the rest of this section we  will prove that $\bf G^{\frak c}$ is close to $\E\br{\cal G}$. The precise statement constitutes Theorem \ref{SecDetEq}.

We first establish some properties of $\nu$ and $\check \nu$.

\begin{lem}\label{tildenutrueproba}
$\check \nu$ is a probability distribution, and $g_{\check \nu} = \frac{\gamma_n - 1} z  +  {\gamma_n}  g_{\nu}$.
\end{lem}

\begin{proof} $\check \nu$ has total mass $1$ and is a positive measure excepted maybe in $\{ 0 \} $ if $\gamma_n > 1$. In this case, from \cite[Theorem 4.1]{belinschi2003atoms} we have $\nu(\{0\})\geq 1 -  {\gamma_n}^{-1}$, thus $\check \nu(\{0\})\geq0$ and $\check \nu$ is a probability measure.

The identity between the Stieltjes transforms is easily verified:
\begin{align*}
g_{\check \nu}(z) &= \int_\R \frac{1}{t-z} \check \nu(dt) \\
    &= \pt{1-\gamma_n} \int _\R\frac{1}{t-z} \delta_0(dt) + \gamma_n \int _\R\frac{1}{t-z}  \nu(dt) \\
    &= \frac{1 - \gamma_n} {-z} + \gamma_n g_\nu(z).
\end{align*}
\end{proof}

Let us recall the definition of $\Omega=(- \infty, z]$ if $z \in \R^{*-}$, and $\Omega = \{ \frak l \in \C$ such that $\Im(\frak l) \geq \Im(z)$ and  $\Im(z^{-1} {\frak l})\geq 0$\} if $z \in \C^+$. We introduce the following functional on $\Omega$:
 \[ \cal F(\frak l) =   z + \frac z n \Tr\pt{\bf G^{\frak l} \Sigma} \]
 
 \begin{prop}\label{cfixedpoint} ${\frak c}$ is a fixed point of $\cal F$ and $\frak c \in \Omega$. Moreover $\frac 1 p \Tr\pt{\bf G^{\frak c}} = g_\nu$.
 \end{prop}
 
 \begin{proof} The fact that $\frak c \in \Omega$  is a consequence of the classical properties of the Stieltjes transforms. Using the properties of the multiplicative free convolution  we have  $g_\nu(z) = \int_\R \frac 1 {\pt{1 - \gamma_n - \gamma_n z g_\nu(z)} t - z}\mu_\Sigma(dt) $. Moreover $\frac z {\frak c} = 1 - \gamma_n - \gamma_n z g_\nu(z)$ using Lemma \ref{tildenutrueproba}. We deduce that $\frac 1 p \Tr\pt{\bf G^{\frak c}} = \int_\R \frac 1 {\frac z {\frak c} t - z} \mu_\Sigma(dt) = g_\nu(z) $.  On the other hand:
\begin{align*}
 \frac z n  \Tr\pt{\bf G^{\frak c} \Sigma} &=  \frac {\gamma_n  } p \Tr \pt{\pt{ \frac \Sigma {\frak c} - I_p}^{-1}\Sigma} \\
 &= \gamma_n \int_\R \frac{ t }{  \frac t {\frak c}-1}\mu_\Sigma(d t)  \\
 &= \gamma_n  \int_\R \frak c \pt{1 + \frac 1 {\frac t {\frak c }-1}} \mu_\Sigma(dt) \\
 &= \gamma_n \frak c \pt{ 1 +  z g_\nu(z)} = \frak c - z \end{align*}
 This proves that $\frak c$ is indeed a fixed point of $\cal F$.\end{proof}

 \begin{rem} We will prove in Corollaries \ref{UniquenessCRealCase} and \ref{UniquenessCComplexCase} that  $\frak c$ is the unique fixed point of $\cal F$.  
 The two characterizations we have for $\frak c$ correspond to the two characterizations we mentioned for $\frak a$ and $\frak b$. 
 Indeed by considering that $\E\br{\cal G_-} \approx \E\br{\cal G} \approx \bf G^{\frak b}$ in the definition of $\frak b$, we can see that $\frak b \approx z + \frac z n \Tr\pt{ \bf G^{\frak b} \Sigma }$. In other terms $\frak b$ is an approximate fixed point of $\cal F$, and we can hope that it will get close to $\frak c$ which is the true fixed point of $\cal F$. 
 
 The definition of $\frak c$ using Stieltjes transforms also admits a heuristic interpretation using the co-resolvent $\check {\cal G}$.  
 We will later understand that $\nu = \mmp\pt{\gamma_n } \boxtimes \mu_{\Sigma}$ is a good approximation of $\mu_K$, the spectral distribution of $K$, in the sense that their Stieltjes transforms are close. On the other hand, the spectra of $K = n^{-1} X X^\top$ and $\check K = n^{-1} X^\top X$ differ by $|n-d|$ zeroes, equivalently $\mu_{\check K} = (1-\gamma_n) \delta_0 + \gamma_n \mu_K$. The measure $ \check \nu = (1-\gamma_n ) \delta_0 + \gamma_n  \nu$ is therefore a good approximation of $\mu_{\check K}$. With the help of concentration arguments, it makes sense that $\frak b^{-1} = \E\br{\frak a}^{-1} =\E\br{- \check {\cal G}_{11} ^{-1}}^{-1} \approx  \E\br{- g_{\check K}}^{-1} \approx - g_{\check \nu}^{-1}$, which enlightens the choice of $\frak c =- g_{\check \nu}^{-1}$.
 \end{rem}

 Our strategy for the rest of this section is as follows: we quantify the error made when saying that $\frak b$ is an approximate fixed point of $\cal F$ in the next proposition. We then study the stability of the fixed point equation to control the gap between $\frak b$ and $\frak c$. This analysis will rely on the fact that $\cal F$ is a contraction mapping, with slightly different arguments in the real and in the complex case. We finally deduce that $\bf G^{\frak c}$ is close to $\bf G^{\frak b}$, itself close to $\E\br{\cal G}$, which proves Theorem \ref{SecDetEq}.

 \begin{prop} \label{b-Fb}$\verti{\frak b - \cal F(\frak b) }\leq  O\pt{n^{-1} \eta^5 |z|^{\frac 5 2} }$.
 \end{prop}
 
 \begin{proof}Let us remember the estimates obtained in the proof of Theorem \ref{FirstDetEq}: $\vertif{\E[\cal G - \cal G_-]}\leq  O\pt{n^{- \frac 1 2} 
|z| \eta^3  }$ and $\vertif{\E[\cal G] - \bf G^{{\frak b} }} \leq   O\pt{n^{-\half} \eta^5 |z|^{\frac 3 2} }$. We deduce that:  \begin{align*}
\verti{\frak b - \cal F( {\frak b}) } 
 &= \verti{ \frac {{z}} n  \Tr\pt{\pt{\E[\cal G_-]- \bf G^{{\frak b}}} \Sigma } }\\
 &\leq \frac{|z|}{ n } p^{\half}  \pt{  \vertif{ \E[\cal G_- - \cal G]} + \vertif{\E[\cal G] - \bf G^{ {\frak b} }}}  \\
  &\leq O\pt{n^{-1} \eta^5 |z|^{\frac 5 2} }.
\end{align*}\end{proof}

\subsection{Stability of $\cal F$ in the real case}

Let us working with $z$ belonging to $\R^{*-}$. In this context we recall that $\Omega = (- \infty, z]$. Our first objective in this section is to prove that $\cal F$ is a contraction mapping around $\frak c$ for the usual distance on $\R$.

\begin{lem} For any $\frak l \in \Omega$, $z - \gamma_n \leq \cal F(\frak l) \leq z \leq 0 $. In particular $1 \leq \frac {\frak c} z \leq 1 + \frac{\gamma_n} {|z|}$.
\end{lem}

\begin{proof}

The matrices $\bf G^{\frak l}$ and $\Sigma$ commute and are both positive semi-definite, hence the  product $\bf G^{\frak l} \Sigma$ is  positive semi-definite. We deduce that $0 \leq \Tr\pt{ \bf G^{\frak l} \Sigma} \leq  p |z|^{-1}$, and 
$- \gamma_n \leq \cal F(\frak l) - z = \frac z n \Tr\pt{ \bf G^{\frak l} \Sigma}  \leq 0$.

Applying these bounds to $\frak c = \cal F(\frak c)$ proves the second assertion.
\end{proof}
\begin{lem} \label{SpecialBoundReal}For any $\frak l \in \Omega$,  $\vertiii{\frac z {\frak l}\bf G^{\frak l} \Sigma} \leq  \frac 1 {1+|z|}$.
\end{lem}
\begin{proof}The eigenvalues of $\frac z {\frak l}\bf G^{\frak l} \Sigma= \pt{ \Sigma - \frak l I_p}^{-1} \Sigma $ are given by $\frac \lambda{\lambda - l}$ where $\lambda$ are the eigenvalues of $\Sigma$. All $\lambda$ belong to $[0,1]$ and $l \leq z < 0$, so we have the following inequalities:\[ 0 \leq  \frac \lambda {\lambda - \frak l} = \frac \lambda {\lambda +|\frak l|}\leq \frac \lambda {\lambda +| z|} \leq \frac 1 {1+|z|} .\] We deduce that all the eigenvalues of $\frac z {\frak l}\bf G^{\frak l} \Sigma$ belong to $\left[0, \frac 1 {1+|z|} \right]$, which proves the bound in spectral norm.
\end{proof}

\begin{prop} $\cal F$ is a contraction mapping around $\frak c$. More precisely, $\cal F$ is $k_{\cal F}$-Lipschitz around $\frak c$ with $k_{\cal F} = \frac{\gamma_n}{(1+|z|)(\gamma_n+|z|)} < 1 $.
\end{prop}

\begin{proof}
For any $\frak l $, using a resolvent identity:
\begin{align*}
   \cal F(\frak c) - \cal F(\frak l) &=  \frac z n \Tr\pt{\pt{\bf G^{\frak c}-\bf G^{\frak l}} \Sigma} \\&=  \frac z n \Tr\pt{\bf G^{\frak c}\pt{ \frac z {\frak l} -\frac z {\frak c}}\Sigma \bf G^{\frak l} \Sigma} \\
   &=  \frac{\frak c -\frak l  } n  \Tr\pt{\frac z {\frak c} \bf G^{\frak c}\Sigma \frac z {\frak l} \bf G^{\frak l} \Sigma}.
   \end{align*}
A straightforward application of the last lemma would only yield a Lipschitz factor of $ {\gamma_n } {(1+|z|)^{-2}}$ which may not be small enough. To go beyond, we exploit the fact that  $\frak c$ is a fixed point of $\cal F$, in particular $1 - \frac z  {\frak c}  = \frac 1 n \Tr\pt{\frac z {\frak c}\bf G^{\frak c}\Sigma}$. All the matrices commute in the above expression, and we can use the bounds of the preceding lemmas to obtain:
  \begin{align*}
      \verti{\cal F(\frak c) - \cal F(\frak l)}&\leq \verti{\frak c -\frak l } \cdot \verti{\frac 1 n \Tr\pt{\frac z {\frak c}\bf G^{\frak c}\Sigma}} \cdot \vertiii{\frac z {\frak l}\bf G^{\frak l} \Sigma} \\
      &\leq \verti{\frak c -\frak l } \cdot \pt{1 - \frac z {\frak c}} \cdot \frac 1 {1+|z|} \\
      &\leq  \verti{\frak c -\frak l } \cdot \pt{1 - \frac 1 {1 + \frac {\gamma_n} {|z|}}}\cdot \frac 1 {1+|z|} \\
      &=\verti{\frak c -\frak l } \cdot \frac{\gamma_n}{(1+|z|)(\gamma_n+|z|)}.
  \end{align*}
   \end{proof}

  \begin{cor} \label{UniquenessCRealCase}
In the real case, ${\frak c}$ is the unique fixed point of $\cal F$.
  \end{cor}

  \begin{proof}
If $\frak l$ is a fixed point of $\cal F$, then $\verti{\frak l - \frak c}=\verti{\cal F(\frak l) - \cal F( \frak c)} \leq k_{\cal F} \verti{\frak l - \frak c}$. Given that $ k_{\cal F} < 1$, necessarily $\verti{\frak l - \frak c} = 0$ and $\frak l$ must be equal to $\frak c$. \end{proof}

 \begin{prop}\label{GapBCRealCase}
 $\verti{\frak b - \frak c}\leq O\pt{n^{-1}  |z|^{- \frac 7 2}}$.
 \end{prop}
 \begin{proof}With $k_{\cal F}= \frac{\gamma_n}{(1+|z|)(\gamma_n+|z|)} \leq \frac 1{1 + |z|}< 1$, by Proposition  
\begin{align*}
   \verti{ \frak b - \frak c} \leq \verti{\frak b - \cal F(\frak b)} + \verti{\cal F(\frak b)  - \cal F(\frak c)} 
   \leq \verti{\frak b - \cal F(\frak b)} + k_{\cal F} \verti{ \frak b - \frak c}.
\end{align*}
Given that $\verti{\frak b - \cal F(\frak b)}\leq {O\pt{n^{-1} |z|^{- \frac 5 2} } }$ by Proposition \ref{b-Fb} and $\frac 1 {1 - k_{\cal F}} \leq \frac{1+|z|}{|z|} \leq O(|z|^{-1})$, we deduce that $\verti{ \frak b - \frak c} \leq \frac { \verti{\frak b - \cal F(\frak b)}}{1 - k_{\cal F}}\leq {O\pt{n^{-1} |z|^{- \frac 7 2} } } $.
\end{proof}

\subsection{Stability of $\cal F$ in the complex case}

We are now working with $z$ belonging to $\C^+$. In this context we recall that $\Omega = \{ \frak l \in \C$ such that $\Im(\frak l) \geq \Im(z)$ and  $\Im(z^{-1} {\frak l})\geq 0\}$. As in the last section, we would like $\cal F$ to be a contraction mapping. To this end we need  to ditch the usual metric on $\C^+$, and work instead with the following semi-metric:
\[ d(\omega_1, \omega_2) = \frac{\verti{\omega_1 - \omega_2}}{\Im(\omega_1)^\half \Im(\omega_2)^\half}.\]
\begin{lem}\label{ImFl} For any $\frak l \in \Omega$, $\Im(z) \leq \Im\pt{\cal F(\frak l)} \leq \Im(z) + \gamma_n |z| \eta$, and:
\begin{align*}
  \Im(\cal F(l)) - \Im  (z) &=  \frac  {|z|^2}  n \frac{\Im ({\frak l})}{|\frak l|^2} \vertif{\Sigma \bf G^{\frak l}}^2 .
\end{align*}
\end{lem}

\begin{proof} Applying the resolvent identity to $z \bf G^{\frak l} = \pt{\frac \Sigma {\frak l}  - I_p}^{-1}$ and $\ov{z \bf G^{\frak l}} =  \pt{\frac \Sigma {\ov{\frak l}}  - I_p}^{-1}$ leads to:
\[\Im ( z \bf G^{\frak l} ) =   \frac {1} {2i}  z \bf G^{\frak l} \left( \frac \Sigma {\ov{\frak l}}  - \frac \Sigma {\frak l} \frac \Sigma {\ov{\frak l}} \right)  \ov{ z\bf G^{\frak l}} = |z|^2 \frac{\Im ({\frak l})}{|\frak l|^2} \bf G^{\frak l} \Sigma \ov{\bf G^{\frak l}} .\]
From there we deduce that:
\begin{align*}
\Im(\cal F(l)) &= \Im(z) + \frac 1 n \Tr\pt{\Sigma  \Im\pt{  z \bf G^{\frak l}}} \\ &= \Im(z) + \frac {|z|^2} n\frac{\Im ({\frak l})}{|\frak l|^2} \Tr\pt{ \Sigma \bf G^{\frak l} \Sigma \ov{\bf G^{\frak l}} } \\
&= \Im(z) + \frac  {|z|^2}  n \frac{\Im ({\frak l})}{|\frak l|^2} \vertif{\Sigma \bf G^{\frak l}}^2.\end{align*}
This identity proves the lower bound on $\Im\pt{\cal F(\frak l)}$. The upper bound is a consequence of: $\verti{\frac 1 n \Tr\pt{\Sigma  \Im\pt{  z \bf G^{\frak l}}} } \leq \gamma_n \cdot \vertiii{\Sigma} \cdot \vertiii{  \Im\pt{  z \bf G^{\frak l}}}   \leq \gamma_n |z| \eta$.
\end{proof}

\begin{prop} \label{FcontractionMap}$\cal F$ is a contraction mapping with respect to $d$. More precisely, $\cal F$ is $k_{\cal F}$-Lipschitz with $k_{\cal F} = \frac{\gamma_n |z| \eta^2}{1+\gamma_n |z| \eta^2} < 1$.
\end{prop}

\begin{proof}
For any $\frak l,\frak l' \in \Omega$, using a resolvent identity:
\begin{align*}
   \cal F(\frak l) - \cal F(\frak l') &=  \frac z n \Tr\pt{\pt{\bf G^{\frak l}-\bf G^{\frak l'}} \Sigma} \\&=  \frac z n \Tr\pt{\bf G^{\frak l}\pt{ \frac z {\frak l'} - \frac z {\frak l} }\Sigma \bf G^{\frak l'} \Sigma} \\
   &= \frac{z^2}n \frac{\frak l-\frak l' }{\frak l  \frak l'} \Tr\pt{\bf G^{\frak l}\Sigma \bf G^{\frak l'} \Sigma}.
   \end{align*}
     We can recognize $d(\frak l,\frak l')= \frac{\verti{\frak l-\frak l'}}{\Im(\frak l)^{\frac 1 2}\Im(\frak l')^{\frac 1 2}}$ and  the expression  for $\Im\pt{\cal F(\frak l)}-\Im(z)$ from the preceding lemma:
     \begin{align*}
   \verti{ \cal F(\frak l) - \cal F(\frak l') }  &\leq  \frac {|z|^2} n \frac{\verti{ \frak l-\frak l'}}{\verti{\frak l}  \verti{\frak l'}} \vertif{\bf G^{\frak l}\Sigma} \vertif{\bf G^{\frak l'}\Sigma}\\
  &= \frac{\verti{ \frak l-\frak l'}}{\Im ({\frak l})^\half\Im ({\frak l'})^\half}\pt{ \frac {|z|^2} n \frac{\Im ({\frak l})}{|\frak l|^2} \vertif{\Sigma \bf G^{\frak l}}^2  }^\half   \pt{ \frac  {|z|^2} n\frac{\Im ({\frak l'})}{|\frak l'|^2} \vertif{\Sigma \bf G^{\frak l'}}^2  }^\half\\ 
  &=d(\frak l, \frak l') \pt{\Im\pt{\cal F(\frak l)}-\Im(z)}^\half\pt{\Im\pt{\cal F(\frak l')}-\Im(z)}^\half.
\end{align*}
 We thus have:
 \begin{align*}
     d(\cal F(\frak l),\cal F(\frak l')) &= \frac{ \verti{\cal F(\frak l) - \cal F(\frak l')}}{{\Im(\cal F(\frak l))}^\half {\Im(\cal F(\frak l'))}^\half} \\
     &\leq d(\frak l,\frak l') 
     \pt{1 - \frac{\Im(z)}{\Im(\cal F(\frak l))}}^{\frac 1 2} 
     \pt{1 - \frac{ \Im(z)}{\Im(\cal F(\frak l'))}}^{\frac 1 2}.\end{align*}
Given that $\Im(z) \leq \Im\pt{\cal F(\frak l)} \leq \Im(z)+ \gamma_n |z| \eta $, we have  $0 \leq {1 - \frac{\Im(z)}{\Im(\cal F(\frak l))}} \leq 1 - \pt{1+ \gamma_n |z| \eta^2}^{-1} = \frac{\gamma_n |z| \eta^2}{1+\gamma_n |z| \eta^2}$, hence $ d(\cal F(\frak l),\cal F(\frak l')) \leq  \frac{\gamma_n |z| \eta^2}{1+\gamma_n |z| \eta^2} d(\frak l,\frak l') $.\end{proof}

\begin{cor} \label{UniquenessCComplexCase}
In the complex case, ${\frak c}$ is the unique fixed point of $\cal F$.
  \end{cor}

  \begin{proof}
In order to obtain a contradiction, let $\frak l \in \Omega$ be a fixed point of $\cal F$ such that $\bf l \neq \bf c$ . Using lemma \ref{ImFl}, $\Im(\frak l) = \Im(\cal F(\frak l))=\Im(z) + \frac  {|z|^2}  n \frac{\Im ({\frak l})}{|\frak l|^2} \vertif{\Sigma \bf G^{\frak l}}^2$. Since $\Im(z)>0$, we have $1 > \frac  {|z|^2}  n \frac{1}{|\frak l|^2} \vertif{\Sigma \bf G^{\frak l}}^2  $. Similarly $1 > \frac  {|z|^2}  n \frac{1}{|\frak c|^2} \vertif{\Sigma \bf G^{\frak c}}^2  $.

As seen in the proof of Proposition \ref{FcontractionMap}, we have $ \frak l - \frak c = \cal F(\frak l) - \cal F(\frak c) = \frac{z^2}n \frac{ \frak l - \frak c}{\frak l  \frak c} \Tr\pt{\bf G^{\frak l}\Sigma \bf G^{\frak c} \Sigma}$. We obtain a contradiction after taking a Cauchy-Schwarz inequality:
\[ 1 = \frac{|z|^2}{n}\frac{1}{|\frak l|| \frak c| } \Tr\pt{\bf G^{\frak l}\Sigma \bf G^{\frak c} \Sigma} \leq \frac{|z|^2}{n}\frac{1}{|\frak l| } \vertif{\bf G^{\frak l}\Sigma} \frac 1 {| \frak c|}\vertif{ \bf G^{\frak c} \Sigma} > 1. \] \end{proof}

Our second objective in this section is to use the contraction property of $\cal F$ to quantify how $\frak b$ and $\frak c$ are close. Since $d$ is not a true metric (it lacks the triangular inequality), we need a few additional lemmas:

\begin{lem}
Given $\omega_1$ and $\omega_2 \in \C^+$, $\verti{\frac{1}{{\Im(\omega_1)}^\half }} \leq \verti{\frac{1}{{\Im(\omega_2)}^\half } }\pt{1 + d(\omega_1,\omega_2)  }$. 
\end{lem}

\begin{proof} It is simply a matter of computing:\begin{align*}
\frac{1}{{\Im(\omega_1)}^\half} &= \frac{1}{{\Im(\omega_2)}^\half}  \pt{ 1  + \frac{{\Im(\omega_2)}^\half-{\Im(\omega_1)}^\half}{{\Im(\omega_1)}^\half}} \\ &= \frac{1}{{\Im(\omega_2)}^\half}  \pt{ 1  + \frac {\Im(\omega_2) -  \Im(\omega_1)} {{\Im(\omega_1)}^\half\pt{{\Im(\omega_2)}^\half + {\Im(\omega_1)}^\half}  }},  \\
\verti{\frac{1}{{\Im(\omega_1)}^\half}} &\leq  \verti{\frac{1}{{\Im(\omega_2)}^\half} } \pt{ 1  +  \frac {\verti{\Im(\omega_2) -  \Im(\omega_1)}} {{\Im(\omega_1)}^\half\pt{{\Im(\omega_2)}^\half + {\Im(\omega_1)}^\half}  }}
\\ 
&\leq  \verti{\frac{1}{{\Im(\omega_2)}^\half} }\pt{1 + d(\omega_1,\omega_2)  }.
\end{align*}
\end{proof}

\begin{lem}\label{StabilityContract} Let $f: \C^+ \to \C^+ $ be $k_f$-Lipshitz with respect to $d$, let $c$ be a fixed point of $f$ and let $b \in \C^+$ be another point.
We let $\Delta = d(b,f(b))$.
Provided ${k_f}(1+\Delta) < 1$, then 
the following inequality holds true: \[\verti{c-b} \leq \frac {\verti{f(b) - b}} { 1 - {k_f}(1+\Delta)} . \]
\end{lem}

\begin{proof} Using the above lemma with $\omega_1 = b$ and $\omega_2 = f(b)$, we have:
\begin{align*}
    \frac{|c-f(b)|}{ {\Im(c)}^\half  {\Im(b)}^\half } \leq \frac{|f(c)-f(b)|}{{\Im(c)}^\half  {\Im(f(b))}^\half }\pt{1 + d(b,f(b))} 
    \leq {k_f}\, d(c,b) (1 + \Delta).
\end{align*}
Therefore:
\begin{align*}
 d(c,b) &= \frac{\verti{c - b}}{{\Im (c)}^\half {\Im (b)}^\half } \\
 &\leq \frac{\verti{c - f(b)}}{{\Im (c)}^\half {\Im (b)}^\half }
 +
 \frac{\verti{f(b) - b}}{{\Im (c)}^\half {\Im (b)}^\half } \\
 &\leq  {k_f} d(c,b) (1 + \Delta) + \frac{\verti{f(b) - b}}{{\Im (c)}^\half {\Im (b)}^\half } .\end{align*}
Provided ${k_f}(1+\Delta) < 1$, we obtain:
 \begin{align*}
d(c,b) &\leq \frac 1 { 1 - {k_f}(1+\Delta)}  \frac{\verti{f(b) - b}}{{\Im (c)}^\half {\Im (b)}^\half }  \\
\verti{c-b} &\leq \frac {\verti{f(b) - b}} { 1 - {k_f}(1+\Delta)}
.\end{align*}
 \end{proof}

We will now prove the equivalent of Proposition  \ref{GapBCRealCase} in the complex case. 

 \begin{prop}\label{GapBCComplexCase}
 If ${\eta^8 |z|^{\frac 7 2}} =o(n)$, then $\verti{\frak b - \frak c}\leq O\pt{n^{-1}\eta^7 |z|^{\frac 7 2}}$.
 \end{prop}

\begin{proof}
We want to apply Lemma \ref{StabilityContract} to the function $\cal F$, the true fixed point ${\frak c}$ and the approximate fixed point $\frak b $. We have:
\begin{align*}
\Delta= d\pt{\frak b, \cal F(\frak b)} &=  \frac{\verti{\frak b - \cal F( {\frak b}) }} {\Im(\frak b)^\half \Im\pt{\cal F(\frak b}^\half} \leq O\pt{n^{-1} \eta^6 |z|^{\frac 5 2} },
\end{align*}
since $\Im( \frak b )$ and  $\Im( \cal F( {\frak b} ))  \geq \Im(z) = \eta^{-1}$. On the other hand $\cal F$ is $k_{\cal F}$-Lipschitz with $k_{\cal F} = \frac{{\gamma_n \eta^2 |z|}}{1 + {\gamma_n \eta^2 |z|}}$. If ${\eta^8 |z|^{\frac 7 2}} =o(n)$, then $ {\gamma_n \eta^2 |z|} \Delta \to 0$ and  for $n$ large enough: \[ k_{\cal F}(1+\Delta) = \frac{{\gamma_n \eta^2 |z|} + {\gamma_n \eta^2 |z|} \Delta}{1+{\gamma_n \eta^2 |z|}}\leq \frac{{\gamma_n \eta^2 |z|} + 0.5}{1 + {\gamma_n \eta^2 |z|}} =  1- \frac 1{2(1+{\gamma_n \eta^2 |z|})} < 1. \]
From Lemma \ref{StabilityContract} we have: $\frac 1 {1 - k_{\cal F}(1+\Delta)} \leq 2\pt{1 + \gamma_n \eta^2 |z|} \leq O(\eta^2 |z|)$. We deduce that:
\begin{align*}
   \verti{{\frak c} - {\frak b}} &\leq \frac {\verti{\cal F( {\frak b}) -  {\frak b}}}{1 - k_{\cal F}(1+\Delta)} \leq O\pt{n^{-1}\eta^7 |z|^{\frac 7 2}}.
\end{align*}
\end{proof}

\subsection{Proof of the second deterministic equivalent }

Let us recall the definitions of $\nu = \mmp\pt{\gamma_n } \boxtimes \mu_{\Sigma}$, $\check \nu = (1-\gamma_n ) \delta_0 + \gamma_n  \nu$, ${\frak c} = -  { g_{\check \nu}}^{-1}$, and $\bf G^{\frak c}=  \pt{\frac z {\frak c} \Sigma - z I_p}^{-1} $. 

\begin{thm}[Second equivalent]\label{SecDetEq} Let $z(n)$ be a sequence of spectral arguments in $\R^{*-}$. If $O(n^{-1}) \leq |z|^7 \leq O(1)$, then:  \[ \vertif{\E[\cal G] - \bf G^{{\frak c}} } \leq O\pt{\frac 1{n^\half |z|^{\frac{11} 2}} } \] 

If $z(n)$ is a sequence of spectral arguments in $\C^+$ such that $\Im(z)$ is bounded and ${\Im(z)^{-16} |z|^{ 7 }} =o(n)$, then:   \[ \vertif{\E[\cal G] - \bf G^{{\frak c}} } \leq O\pt{\frac{ |z|^{\frac 5 2}}{n^{\half }\Im(z)^9}}  \] 
\end{thm}

\begin{rem} In the real case, the assumptions on $z$ are not strictly necessary to obtain an explicit bound. Indeed we can always find a $O\pt{n^{-\half}|z|^{-2}\kappa}$ estimate with the same parameters as in the remark following Theorem \ref{FirstDetEq}: $\sigma^2 = |z|^{-2}\pt{1 + n^{-1} |z|^{-1}}$ and $\kappa = \sigma^2 |z|^{-1} + \sigma |z|^{- \frac 5 2}$.

The situation in the complex case  is however different this time, because we really need ${\Im(z)^{-16} |z|^{ 7 }} =o(n)$ to prove the stability of the fixed point problem. This is a stronger hypothesis than the hypothesis $\Im(z)^{-10}|z|^3 \leq O(n)$ we used for mere convenience in the first deterministic equivalent.

Also note that the exponent in the real case is slightly better than the one we would obtain by replacing $\Im(z)$ in the complex setting  by $|z|$ in the real setting  ($11/2$ instead of $13/2$). Surprisingly enough, this lack of symmetry does not come from the different methods we use to analyze the stability of the fixed point equation, but rather from a difference in the Lipschitz parameter of the application $\frak l \in \Omega \mapsto \bf G^{\frak l}$ as we will see in the upcoming proof.
\end{rem}

\begin{proof} We have quantified how close $\frak b$  and $\frak c$ are in Propositions \ref{GapBCRealCase} and \ref{GapBCComplexCase}, and we can derive a Lipschitz property for $\frak l \mapsto \bf G^{\frak l}$ using a resolvent identity:
\begin{align*}
     \vertif{\bf G^{{\frak l}} - \bf G^{{\frak l'}}} = \vertif{z\verti{\frak l - \frak l'}
     \frac{\bf G^{{\frak l'}}}{{\frak l'}}
     \Sigma
     \frac{\bf G^{{\frak l}}}{{\frak l}} } \leq p^\half \frac{\verti{\frak l-\frak l'}}{|z|} \vertiii{\frac z {\frak l'} \bf G^{ {\frak l'}}\Sigma}\cdot \vertiii{\frac z {\frak l} \bf G^{ {\frak l}}}.
\end{align*}
This is where we can get better exponents in the real case using Lemma \ref{SpecialBoundReal}. $\vertiii{\frac z {\frak b} \bf G^{ {\frak b}}\Sigma } \leq \frac 1 {1+|z|} \leq 1$, thus:
\begin{align*}
     \vertif{\bf G^{{\frak c}} - \bf G^{{\frak b}}}      &\leq p^\half |z|^{-1}O\pt{n^{-1}|z|^{- \frac  7 2 }}|z|^{-1} \leq O\pt{n^{-\half} |z|^{- \frac {11} 2 }}.
\end{align*}
In the complex case, we only have $\vertiii{\frac z {\frak b} \bf G^{ {\frak b}}\Sigma } \leq \eta$ by Proposition \ref{PropertiesAandB}, thus:
\begin{align*}
     \vertif{\bf G^{{\frak c}} - \bf G^{{\frak b}}} 
     &\leq p^\half |z|^{-1}O\pt{n^{-1}\eta^7 |z|^{\frac 7 2}}\eta^2  \leq O\pt{n^{-\half} \eta^9 z^{\frac 5 2}}
\end{align*}
In both cases, the first deterministic equivalent in Theorem \ref{FirstDetEq} gives a term of lesser magnitude for $\vertif{   \E[\cal G] - \bf G^{ {\frak b}}}$, which completes the proof since:
 \begin{align*}
 \vertif{ \E[\cal G] - \bf G^{{\frak c}}}\leq  \vertif{   \E[\cal G] - \bf G^{ {\frak b}}} +  \vertif{  \bf G^{\frak c} - \bf G^{ {\frak b}}}\end{align*}. \end{proof}

\section{Proof of the main results}

Most of the technical work is already achieved in the sections 4-6. For instance, the Proposition \ref{ConcentrationGK} stating the concentration of $\cal G_K(z)$ is a direct consequence of Corollary \ref{ConcentrationG}. Indeed with $\eta = |z|^{-1}$ if $z \in \R^{*-}$, we have  $  n^{-\half} \eta^2 |z|^\half{  } = n^{- \half} |z|^{- \frac 3 2} $. With $\eta = \Im(z)^{-1} $ if $z \in \C^+$  we have $  n^{-\half} \eta^2 |z|^\half{  } = n^{-\half} \Im(z)^{-2} |z|^\half$.

The conclusions of our main theorem \ref{MainThm} are exactly the same as in Theorem \ref{SecDetEq} after carefully comparing the definitions of the deterministic equivalent.
We need however to add a quick argument for the Proposition \ref{DetEq} and the Corollary \ref{CorDetEq}.

\begin{proof}[Proof of Proposition \ref{DetEq}]
From \ref{MainThm}, uniformly in any $A \in \R^{p \times p}$ we have:\begin{align*}
 \verti{\Tr\pt{\E[\cal G_K(z)] A - \bf G(z) A}} &\leq \vertif{\E[\cal G_K(z)]  - \bf G(z)} \cdot \vertif{A} \\
 &\leq \vertif{A} \, O\pt{\frac {\kappa} {n^\half} } ,\end{align*}where $\kappa = \frac 1 {|z|^{\frac {11} 2}}$ in the real case, or $\kappa = \frac{ |z|^{\frac 5 2}}{\Im(z)^9}$ in the complex case. On the other hand, from Propositions \ref{ConcentrationGK}, \ref{LipschitzTransformsConcentration} and \ref{VarianceConcentration}, uniformly in any $A \in \R^{p \times p}$ we have:\[  \verti{\Tr\pt{\cal G_K(z) A - \E[\cal G_K(z)] A}} \leq \vertif{A} \, O\pt{\frac {\tau (\log n)^\half} {n^\half} } \rm{\quad a.s. }  \]
where $\tau = \frac 1 { |z|^{\frac 3 2}} $ in the real case, or $\tau = \frac {|z|^\half} {\Im(z)^2 }$ in the complex case.

In both scenarios, $\tau \leq O(\kappa)$. Indeed in the real case, $\tau \kappa^{-1} = |z|^4 \leq O(1)$. In the complex case, $\tau \kappa^{-1} = \Im(z)^7 |z|^{-2} \leq \Im(z)^5 \leq O(1)$. We deduce that uniformly in any $A \in \R^{p \times p}$:
\begin{align*}
  \verti{\Tr\pt{\cal G_K(z) A - \bf G(z) A}}   &\leq \verti{\Tr\pt{\cal G_K(z) A - \E[\cal G_K(z)] A}} + \verti{\Tr\pt{\E[\cal G_K(z)] A - \bf G(z) A}} \\
    &\leq \vertif{A} \, O\pt{\frac {\kappa(\log n)^\half} {n^\half} } \rm{\quad a.s. } 
\end{align*}

\end{proof}

\begin{proof}[Proof or Corollary \ref{CorDetEq}] Considering $A = p^{-1} I_p$ which satisfies $\vertif{A} = p^{-\half} = O(n^{-\half}) $, and using the identity $p^{-1}  \Tr \bf G(z) = g_\nu(z)$ from Proposition \ref{cfixedpoint} yield:
\begin{align*} \verti{\frac 1 p \Tr \cal G_K(z) - \frac 1 p \Tr \bf G(z) } &= \verti{g_K(z) - g_\nu(z)} \leq O\pt{\frac {\kappa (\log n)^\half} {n} }  \rm{\quad a.s. } 
\end{align*} 
Uniformly in $\bf u \in \Ss^{p-1}$ , with $A = \bf u \bf u^\top $, $\vertif{A} = \vertii{\bf u} = 1 $:
\begin{align*} \verti{\Tr\pt{ \cal G_K(z) \bf u \bf u ^\top-  \bf G(z) \bf u \bf u ^\top }} &=  \verti{g_{\mu_{K,\bf u}}(z) - \bf u^\top \bf G(z) \bf u} &\leq O\pt{\frac {\kappa (\log n)^\half} {n^\half} }    \rm{\quad a.s. }
\end{align*} 

Finally the third assertion is a consequence of the polarization identity for symmetric matrices, which relates the entries of $M$ to terms like $\bf u^\top M \bf u$:
\[ 2 M_{ij} = (\bf e_i + \bf e _j)^\top M (\bf e_i + \bf e _j)- \bf e_i ^\top M \bf e_i - \bf e_j ^\top M \bf e_j. \]
\end{proof}

\newpage

\section{Bounds in Kolmogorov distance}

The goal of this section is to prove Theorem \ref{BoundKolmogorov}, which provides an asymptotic bound in Kolmogorov distance given a quantitative bound on the Stieltjes transforms of probability distributions. It closely follows the ideas of \cite[Theorem 5.2]{banna2020holder}, with some key differences. First of all we do not ask for estimates of the Stieltjes transforms on the whole complex upper plane, but rather on specific sequences that will appear in the proof. This broader setting will be well adapted to the conclusions of Theorem \ref{MainThm}. Compared to the work of \cite{banna2020holder}, we also do not require weak convergence of the distributions, and we work without uniform bounds on the support of our distributions. More precisely, we open the possibility that their second moments grow slowly to $\infty$.

Let us remind the notations $\cal F_\nu$ for the cumulative distribution function (CDF) of $\nu$, and $\Delta(\nu,\mu) = \sup_{t \in \R} \verti{\cal F_\nu(t) - \cal F_\mu(t)}$ for the Kolmogorov distance between $\nu$ and $\mu$. 
The following well-known facts are well-surveyed in \cite{geronimo2003necessary}:  the convergence in Kolmogorov distance implies the weak convergence for probability measures, and there is even an equivalence if the limiting measure admits a Hölder continuous CDF.

\subsection{Bound for fixed measures}

To compare distributions in Kolmogorov distance we will first need a few technical lemmas for fixed probability measures $\mu$ and $\nu$ on $\R^+$. We start with the so-called Bai's Inequality \cite[Theorem 2.1]{bai2008large}:

\begin{lem}\label{BaiIn}
Let $\frak a = \tan\pt{\frac {3 \pi} 8 }$. If $\int_{\R} \verti{  \cal F_\nu(t) - \cal F_\mu(t)}dt< \infty$, then for any $y > 0$:
\begin{align*} \Delta(\mu,\nu) &\leq \frac 2 \pi \pt{ \int_{\R} \verti{g_\nu-g_\mu}(t + i y) dt + \frac 1 y \sup_{x \in \R} \int_{[\pm 2 y \frak a]} \verti{\cal F_{\mu}(x+t)-\cal F_{\mu}(x)} dt } \end{align*}
\end{lem}

The following tail bound for the Stieltjes transforms is a consequence of the intermediate result \cite[Lemma 5.5]{banna2020holder}:

\begin{lem}[]\label{TailStielt} If
$\mu$ and $\nu$ have finite second moments, smaller than $\sigma \geq 1$, then for any $y \in (0,1)$ and $A>0$ we have: \[ \int_{(\pm A)^c} \verti{g_\nu-g_\mu}(t + i y) dt \leq \frac{4 \sigma^{3}} {y^2 A } \]
\end{lem}

Finally the last lemma is a simple consequence of the Hölder property:
\begin{lem}\label{TailCDF} If $ \cal F_{\mu}$ is Hölder continuous with constant $C$ and exponent $\beta \in (0,1]$, then for any $y>0$:
\[  \frac 1 y \sup_{x \in \R} \int_{[\pm 2 y \frak a]} \verti{\cal F_{\mu}(x+t)-\cal F_{\mu}(x)} dt  \leq 50 C y^\beta \]
\end{lem}

Cutting the first integral of Lemma \ref{BaiIn} in $\pm A$ and using the lemmas gives the following estimate for the Kolmogorov distance between two fixed measures:

\begin{prop} \label{FixedBound} Suppose that $\mu$ and $\nu$ have finite second moments, smaller than $\sigma \geq 1$,  that $\int_{\R} \verti{  \cal F_\nu(t) - \cal F_\mu(t)}dt< \infty$, and  that $\cal F_\mu$ is  Hölder continuous with constant $C$ and exponent $\beta \in (0,1]$.  

There is a new constant $C'>0$, only depending on $C$, such that for any $y \in (0,1)$ and $A>0$:\begin{align*} \Delta(\mu,\nu)  \leq C' \pt{A \max_{t \in (\pm A)} \verti{g_\nu-g_\mu}(t + i y)  + \frac{ \sigma^3} {y^2 A } +  y^\beta} . \end{align*}
\end{prop}

\subsection{Asymptotic bound}

Let us consider $\mu_n$ and $\nu_n$ two sequences of probability measures on $\R^+$ satisfying:
\begin{ass} \label{AssKolmo} 
\begin{enumerate}[(1)]
\item $\int_{\R} \verti{  \cal F_{\nu_n}(t) - \cal F_{\mu_n}(t)}dt< \infty$ for all $n$ (a uniform bound is not required).
    \item All CDFs $\cal F_{\mu_n}$ are uniformly Hölder continuous with the same exponent $\beta \in (0,1]$ and constant $C > 0$.
    \item $\mu_n$ and $\nu_n$ have finite second moments bounded by a sequence $(\sigma_n) \geq 1$.
    \item There are exponents $k, l \in (0,\infty)$ and a sequence $(\epsilon_n) >0$ with $(\sigma_n^{3(1+l)} \epsilon_n) \to 0$ such that, for any $z \in \C^+$: \[\verti{g_{\mu_n}(z) - g_{\nu_n}(z)} \leq  \epsilon_n {\frac{|z|^l}{(\Im z)^k}}. \]
\end{enumerate}
This last assumption may be replaced by the following weaker version:
\begin{enumerate}[(4')]\setcounter{enumi}{2}
     \item There are exponents $k, l \in (0,\infty)$, and a sequence $(\epsilon_n) >0$ with $(\sigma_n^{3(1+l)} \epsilon_n) \to 0$ such that, setting $     y_n = \pt{\sigma_n^{3 (1+l)} \epsilon_n}^{\frac{1}{2+2l+k+(2+l)\beta}}$ and $      A_n = \pt{ \frac{\sigma_n^3 }{\epsilon_n } }^{\frac 1 {2 + l}} y_n^{\frac {k-2} {2 + l}}$,
the following inequality holds : \[ \sup_{z  = t + i y_n,\, t \in (\pm A_n)} \verti{g_{\mu_n}(z) - g_{\nu_n}(z)} \leq \epsilon_n \frac{|A_n + i y_n|^{l}}{  y_n ^k}. \]\end{enumerate}

\end{ass}

\begin{thm}\label{BoundKolmogorov} Under the Assumptions \ref{AssKolmo}:
\[ \Delta(\mu_n,\nu_n) \leq 
O\pt{\pt{\sigma_n^{3 (1+l)} \epsilon_n}^{{\frac{\beta}{2+2l+k+(2+l)\beta}}}}. \]
\end{thm}

\begin{proof} The idea is to optimize the inequality given by Proposition \ref{FixedBound} in both $y$ and $A$. We pass on the details leading to the aforementioned choices of $y_n$ and $A_n$. We claim that $ \epsilon_n \frac{A_n^{1+l}}{y_n^k} = \frac{\sigma_n^3}{y_n^2 A_n} = y_n^\beta$. This fact is related to the properties of optimizing sequences and can also  be checked manually. Indeed $A_n^{2+l} = { \frac{\sigma_n^3 }{\epsilon_n } } y_n^{k-2} \iff \epsilon_n \frac{A_n^{1+l}}{y_n^k} =  \frac{\sigma_n^3}{y_n^2 A_n} $, and:
\begin{align*}
 \frac{\sigma_n^3}{y_n^2 A_n}   
=     \sigma_n^3 y_n^{-2} \pt{ \frac{\sigma_n^3 }{\epsilon_n } }^{\frac {-1} {2 + l}} y_n^{\frac {2-k} {2 + l}}  =\pt{ \sigma_n^{3(1+l)} \epsilon_n}^{\frac 1 {2+l} } y_n^{- \frac{2+2l+k}{2+l}}  = y_n^\beta.\end{align*}
Also note that $y_n \to 0$, and $A_n \geq {y_n^{-2 - \beta}} \to \infty$, in particular $(y_n) \in (0,1)$ and $A_n\geq 1$ for $n$ large enough. From Assumption \ref{AssKolmo} (4'):\begin{align*}
      \max_{t \in (\pm A_n)} \verti{g_{\nu_n}(t+i y_n) - g_{\mu_n}(t + i y_n)} &\leq  \epsilon_n \frac{|A_n + i y_n|^{l}}{  y_n ^k} \leq O\pt{\epsilon_n \frac{A_n^l}{y_n^k}}.
\end{align*}
We finally apply Proposition \ref{FixedBound}:
\begin{align*}
 \Delta_(\mu_n,\nu_n)  &\leq C' \pt{A_n \max_{t \in (\pm A_n)} \verti{g_{\nu_n}-g_{\mu_n}}(t + i y_n)  + \frac{ \sigma_n^3} {y_n^2 A_n } +  y_n^\beta}\\
 &\leq C'\pt{O\pt{\epsilon_n \frac{A_n^{1+l}}{y_n^k}} + \frac{ \sigma_n^3} {y_n^2 A_n } +  y_n^\beta} \\
  &\leq  O\pt{ y_n^\beta } \leq 
O\pt{\pt{\sigma_n^{3 (1+l)} \epsilon_n}^{{\frac{\beta}{2+2l+k+(2+l)\beta}}}}.
\end{align*}
\end{proof}

\subsection{Application to the ESD of sample covariance matrices}

Let us go back to the setting of our main result. To apply Theorem \ref{BoundKolmogorov} given the bound of Corollary \ref{CorDetEq}, we need to address two difficulties: the Hölder property for the CDF of $\nu_n = \mmp(\gamma_n) \boxtimes \mu_\Sigma$, and a control on the second moment of $\mu_K$.

\begin{lem} If $\Sigma$ is invertible and   $\vertiii{\Sigma^{-1}}$ is bounded, then all CDF of $ \nu_n$ are uniformly $1/2$-Hölder continuous on $\R^+$. If moreover $\gamma_n \leq 1$, the same property holds on $\R$.
\end{lem}

\begin{proof} $\nu_n = \mmp(\gamma_n) \boxtimes \mu_\Sigma$ is supported on $\br{0,(1+\gamma_n^\half)^2 \lambda_{\max}(\Sigma) }$ and admits a continuous density on $f_\nu$ on $(0,\infty)$ that satisfies  $f_\nu(t) \leq \pi^{-1} \pt{\lambda_{\min}(\Sigma) t \gamma_n}^{-\half}$ (\cite[Lemma 2.3]{bai2012convergence}). Given the bounds on $\gamma_n$ and the eigenvalues of $\Sigma$, we deduce that all $\cal F_{\nu_n}$ are uniformly $1/2$-Hölder continuous on $\R^+$. In the case $\gamma_n \leq 1$, all $\cal F_{\nu_n}$ are identically equal to $0$ on $(-\infty,0]$, hence the Hölder continuity on $\R$.
\end{proof}
\begin{lem} 
 $\vertiii{K} \leq O\pt{1}$ a.s. In particular $\mu_K$  have uniformly bounded support a.s. and the second moments of $\mu_K$ are uniformly bounded a.s.
\end{lem}
\begin{proof}
From Proposition \ref{ConcentrationSpectralNorm} we have $\vertiii{X-\E[X]} \leq O\pt{n^\half}$ a.s. Given that $\vertiii{\E[X]} = n^{\half} \vertii{\E[x]} = O(n^{\half})$, we also have $\vertiii{X} \leq O\pt{n^\half}$ a.s. As a consequence, $\vertiii{K} = n^{-1} \vertiii{X}^2 \leq O\pt{1}$ a.s. which in turn implies that $\mu_K$  have uniformly bounded support and uniformly bounded second moments a.s. \end{proof}

\begin{proof}[Proof of Theorem \ref{ApplicationKolmo}] Let us first treat the case where all $\gamma_n \leq 1$. We want to apply Theorem \ref{BoundKolmogorov} with the measures $\mu_K$ and $\nu$.

The Assumptions (1), (2) and (3), with $\beta = 1/2$ and $\sigma_n = 1$, are consequences of the previous lemmas. Let us explain why the Assumption (4'), with $l=k=9$ and $\epsilon_n = n^{-1}(\log n)^\half$ holds true. $y_n =  \epsilon_n^{\frac{1}{2+2l+k+(2+l)\beta}}=\epsilon_n^{2/69}$ is bounded and $      A_n = y_n^{-2-\beta} = \epsilon_n^{- 5/69}$, hence  $\frac{|A_n+i y_n|^7}{y_n^{16}} \leq O\pt{\epsilon_n^{- 67/69}}=o(n)$. All $z = t + i y_n$, where $t \in (\pm A_n)$, satisfy the Assumptions of Corollary \ref{CorDetEq}, which is why: \begin{align*} \max_{t \in (\pm A_n)} \verti{g_{\nu_n}(t+i y_n) - g_{\mu_n}(t + i y_n)} &\leq   \max_{t \in (\pm A_n)}
      \epsilon_n \frac{\verti{t + i y_n}^{\frac 5 2}}{y_n^9}  \leq O\pt{\epsilon_n \frac{A_n^{9}}{y_n^9}} . \end{align*} We conclude that$\Delta(\mu_K, \nu) \leq O\pt{\epsilon_n^{1/69}} \leq O\pt{n^{-1/70}}$.
      
In the case where all $\gamma_n \geq 1$, $\nu$ has an atom at $0$ (see Lemma \ref{tildenutrueproba}) and $\cal F_\nu$ is even discontinuous. We work instead on $\tilde \nu$, thanks to the following formulas:
\begin{align*}
    \mu_{\check K} = (1-\gamma_n) \delta_0 + \gamma_n \mu_K, \qquad & \qquad 
    \check \nu = (1-\gamma_n) \delta_0 + \gamma_n \nu ,
 \\
    g_{\check K} = \frac{\gamma_n - 1} z  +  {\gamma_n}  g_{K}
 ,\qquad & \qquad
g_{\check \nu} = \frac{\gamma_n - 1} z  +  {\gamma_n}  g_{\nu}, \\
\cal F_{ \mu_{\check K}} = (1-\gamma_n) \bf 1_{\R^+} + \gamma_n \cal F_\mu ,\qquad & \qquad \cal F_{ \check \nu} = (1-\gamma_n) \bf 1_{\R^+} + \gamma_n \cal F_\nu.
\end{align*} $\check \nu$ has no atom in $0$, hence all $\cal F_{\check \nu}$ are identically equal to $0$ on $(-\infty,0]$. Given that they also are uniformly $\half$-Hölder continuous on $\R^+$, they are uniformly $\half$-Hölder continuous on $\R$. Since ${\verti{g_{\check K} - g_{\check \nu}}} = \gamma_n {\verti{g_{ K} - g_{ \nu}}}$, we can apply Theorem \ref{BoundKolmogorov} with the measures $\mu_{\check K}$ and $\tilde \nu$. We conclude that $\Delta(\mu_{K}, \nu) = \gamma_n^{-1} \Delta(\mu_{\check K}, \check \nu)  \leq O\pt{n^{- 1/70}}$.

In all generality, we apply the arguments on the subsequences corresponding to $\gamma_n \leq 1$ and $\gamma_n > 1$ and we still retrieve the same result.
\end{proof}
\begin{rem} In the application of Theorem \ref{BoundKolmogorov}, choosing indices $l=5/2$ and $k = 9$ as in Corollary \ref{CorDetEq}, would lead to $\frac{A_n^7}{y_n^{16}} \leq O\pt{\epsilon_n}^{- 143/73 } \nleq o(n)$ which does not match the hypothesis of Corollary   \ref{CorDetEq}. We could have chosen however any $l=(43+ \epsilon)/5 $, which yields $\Delta(\mu_K, \nu) \leq O(n^{- 1/67 - \epsilon' })$. Given that our quantitative bound is not optimal in the first place, we choose not to include this technical detail in the result. 
\end{rem}

\newpage
To prove Corollary \ref{ApplicationKolmoCV} we need a bound on the Kolmogorov distance between Marčenko-Pastur distributions with different shape parameters. Surprisingly enough, we could not find any references for this question, which is why we propose our own proof below.

\begin{lem}
If $\gamma$ and $\gamma' > 0$, then $\Delta\pt{\mmp(\gamma) , \mmp(\gamma')} \leq \frac{\verti{\gamma-\gamma'}}{\max(\gamma,\gamma')} $.
\end{lem}

\begin{proof} 

Let  $n$ and  $p \geq p'$ be fixed integers, and consider a random matrix $Y \in \R^{p \times n}$ filled with $ \iid(\cal N)$ entries. We set $Y' := \pt{Y_{ij}}_{1 \leq i \leq p',\,1 \leq j \leq n}$ to  be the first $p' \times n$ block matrix of $Y$, and we define the sample covariance matrices $C:=n^{-1}YY^\top \in \R^{p \times p}$ and $C':=n^{-1} Y' {Y'}^\top \in \R^{p' \times p'}$.
We denote by $\cal F := \cal F_{\mu_C}$ and $\cal F' := \cal F_{\mu_{C'}}$ the CDF corresponding with the ESD of $C$ and $C'$. If $\lambda_1 \leq \dots \leq \lambda_p$ are the ordered eigenvalues of $C$ and $\lambda'_1 \leq \dots \leq \lambda'_{p'}$ those of $C'$, then $\cal F(t) = (1/p) \cdot\sum_{k=1}^p \bf 1_{\lambda_k \leq t}$, and similarly for $\cal F'$.

$C'$ being equal to the first $p' \times p'$ block matrix of $C$, the interlacing eigenvalues theorem states that $\lambda_k \leq \lambda'_k \leq \lambda_{k+p-p'}$ for any $k \in \lint 1 , p' \rint$. Depending on the values of $t \in \R$, we can distinguish between three cases:
\begin{enumerate}
    \item If $\cal F'(t) =0$, then $t < \lambda'_1\leq \lambda_{1 + p - p'}$, hence $\cal F(t) \leq \frac{p-p'} p$ and $\verti{\cal F(t)-\cal F'(t)} \leq \frac{p-p'}{p}$.
 \item If $\cal F'(t) =1$, then $t \geq \lambda'_{p'} \geq \lambda_{p'}$, hence $\cal F(t) \geq \frac{p'} p$ and $\verti{\cal F(t)-\cal F'(t)} \leq \frac{p-p'}{p}$.
    \item If $\cal F'(t)=k/p'$ for some  $k \in \lint 1 , p'-1 \rint$, then $ \lambda_k \leq \lambda'_k \leq t < \lambda'_{k+1} \leq \lambda_{k+1+p-p'}$, in particular $\frac k p \leq \cal F(t)  \leq  \frac {k+p-p'} p$. Given that:
    \begin{align*}
        \verti{\frac k p - \frac k {p'}}&=\frac{k(p-p')}{pp'} \leq \frac {p-p'} p\quad,\\
        \verti{\frac {k+p-p'} p - \frac k {p'}}&=\frac{(p-p')(p'-k)}{pp'} \leq \frac {p-p'} p\quad,
    \end{align*}in this case we also have $\verti{\cal F(t)-\cal F'(t)} \leq \frac{p-p'}{p}$. 
\end{enumerate} We have therefore obtained: $
    \Delta\pt{\mu_C ,\mu_{C'}} = \sup_{t \in \R} \verti{\cal F(t) - \cal F'(t)} \leq \frac{p-p'}p$.
    
Consider now parameters  $\gamma \geq \gamma' > 0$, and  $p(n)$ and $p'(n)$ sequences of integers such that $p/n \to \gamma$ and $p'/n \to \gamma'$ respectively. In this case $\limsup_{n \to \infty} \Delta(\mu_C,\mu_{C'}) \leq \lim_{n \to \infty} \frac{p-p'}p = \frac{\gamma-\gamma'}\gamma $. On the other hand $\mu_C$ and $\mu_{C'}$ converge respectively to $\mmp(\gamma)$ and $\mmp(\gamma')$ in Kolmogorov distance (\cite{BSBook}). The lemma is proven after taking the limit in the following triangular inequality:
\begin{align*}
    \Delta(\mmp(\gamma),\mmp(\gamma')) &\leq \Delta(\mmp(\gamma),\mu_C) + \Delta(\mu_C,\mu_{C'}) + \Delta(\mu_{C'},\mmp(\gamma')) .
\end{align*}
\end{proof}
\begin{proof}[Proof of Corollary \ref{ApplicationKolmoCV}] The a.s. weak convergence of $\mu_K$ towards $\nu_\infty$ is a consequence of the regularity of the free convolution with respect to the weak convergence of measures. Moreover from \cite[Proposition 4.13]{bercovici1993free} we have the quantitative bound: $\Delta(\nu_n,\nu_\infty) \leq \Delta\pt{\mmp(\gamma_n),\mmp(\gamma_\infty)} + \Delta(\mu_\Sigma,\mu_\infty) $. Also note that $\frac{\verti{\gamma_n-\gamma_\infty}}{\max ( \gamma_n,\gamma_\infty)}\leq O\pt{\verti{\gamma_n - \gamma_\infty}}$. Using the preceding lemma, we get: 
\begin{align*}
     \Delta (\mu_K, \nu_\infty) \leq O\pt{n^{- \frac 1 {70}}}  + O\pt{\verti{\gamma_n - \gamma_\infty}} + \Delta(\mu_\Sigma,\mu_\infty) .
\end{align*}

\end{proof}

\newpage

\section{Application to Kernel Methods}

This section deals with the regularization of Random Feature Models, using our main result to obtain Theorem \ref{ApplicationKernel}. Let us introduce this method, closely following the approach of \cite{jacot2020implicit}.

\subsection{Kernel Ridge Regression} 
In Kernel Ridge Regression (KRR) we are considering a training dataset made of $N$ distincts vectors $\cal X_1, \dots, \cal X_N \in \R^D$, and  associated real labels $\cal Y = (\cal Y_1,\dots,\cal Y_N)\in \R^N$. 
A kernel $\cal K: {\R^D \times \R^D} \to \R$ is a function such that the matrix $\cal K_{\cal X}:=  \pt{\cal K(\cal X_i,\cal X_j)}_{1\leq i,j \leq N}$ is definite positive. We define the functions $\cal K(\cdot ,\cal X_j ): \R^D \to \R$, $x \mapsto \cal K(x,\cal X _j) $ and $\cal K(\cdot ,\cal X ): \R^D \to \R^N, x \mapsto \pt{\cal K(x,\cal X _j)}_{1 \leq j \leq N}$.

The goal of KRR is to find in the linear span of $\{\cal K(\cdot ,\cal X _j ),  1\leq j \leq N\}$ a function $f$ such that $f(\cal X_i) \approx \cal Y_i$  for all $1 \leq i\leq N$. More precisely, given a ridge parameter $\lambda > 0$, we want to minimize in $\theta  \in \R^N$ the Mean Square Error with a $\lambda$ ridge penalization term:
\[ \frac 1 N \sum_{1\leq i \leq N} \pt{ \sum_{1\leq j \leq N} \theta _j \cal K(\cal X_i,\cal X_j) - \cal Y_i}^2 + \frac \lambda N \theta ^\top \cal K(\cal X,\cal X) \theta .\]

Applying basic linear algebra techniques to the above minimization problem, it can be shown that $\hat \theta := \pt{\cal K_{\cal X} +\lambda I_N}^{-1}  \cal Y$ is the optimal value for $\theta$, and that: \begin{align*}
    \hat f_\lambda^{ (K)}: \R^D &\to \R ,\\
    x &\mapsto \cal K(x,\cal X)\pt{\cal K_{\cal X} +\lambda I_N}^{-1} \cal Y
\end{align*} is the optimal function, thus called  KRR predictor with ridge $\lambda$.

\subsection{Random Features Method}

A set of Random Features (RF) associated with the kernel $\cal K$ is a collection $\Phi = \pt{\Phi^{(j)}}_{1 \leq j \leq P}$ of $P$ random i.i.d. processes $\Phi^{(j)}: \R^D \to \R$, chosen such that they are centered and admit $\cal K$ as covariance function: for all $x,x' \in \R^D$,  $\E\br{\Phi^{(j)}(x)} = 0$ and $\E\br{\Phi^{(j)}(x)\Phi^{(j)}(x')} = \cal K(x,x')$.

In the RF method, we would like to find in the linear span of $\Phi$ a function $f$ such that $f(\cal X_i) \approx \cal Y_i$  for all $1 \leq i\leq N$. More precisely, given a ridge parameter $\lambda > 0$, we want to minimize in $\theta \in \R^P$ the Mean Square Error with a $\lambda$ ridge penalization term:
\[ \frac 1 N \sum_{1\leq i \leq N}  \pt{ \sum_{1\leq j \leq N} P^{-\half}  \theta_j \Phi^{(j)}(\cal X_i) - \cal Y_i}^2 + \frac \lambda N \theta^\top \theta .\]
Let us define the data matrix $F = P^{-\half} \pt{\Phi^{(j)}(\cal X_i) }_{1\leq i \leq N, 1 \leq j \leq P} \in \R^{N \times P}$. Again using linear algebra techniques, it can be shown that: \[\hat \theta :=F^\top  \pt{  FF^\top +\lambda I_N}^{-1}   \cal Y\] is the optimal value for $\theta$, and that:\begin{align*}
    \hat f_\lambda^{(RF)}: \R^D &\to \R \\
    x &\mapsto  P^{-\half} \Phi(x)  F^\top  \pt{  FF^\top +\lambda I_N}^{-1}  \cal Y
\end{align*}is the optimal function, thus called RF predictor with ridge $\lambda$.

\subsection{Effective Ridge Parameter}

The RF predictor is a good approximation of the KRR predictor in the overparametrized regime $P \gg N$. Indeed: 
\[ \hat f_\lambda^{(RF)}(x) = \pt{\Psi(x , \cal X_j)}_{1 \leq j \leq N }\br{\pt{\Psi(\cal X_i , \cal X_j)}_{1 \leq i,j \leq N } + \lambda I_N}^{-1}\cal Y ,\] where $\Psi$ is the random function $(x,x')\mapsto \frac 1 P \sum_{1 \leq k \leq P} \Phi^{(k)}(x)\Phi^{(k)}(x')$.
 When $N$ is fixed and $P \to \infty$, $\Psi$ converges to $\cal K$ by the law of large numbers, hence: \[    \hat f_\lambda^{(RF)} (x) \to \pt{\cal K(x , \cal X_j)}_{1 \leq j \leq N }\br{\pt{\cal K(\cal X_i , \cal X_j)}_{1 \leq i,j \leq N } + \lambda I_N}^{-1}\cal Y =   \hat f_\lambda^{(K)} (x) .\]

The under-parametrized regime $P < N$ is more interesting for practical purposes, since computing the RF predictor requires to invert a $P \times P$ matrix instead of a $N \times N$ matrix for the KRR predictor. In this
regime unfortunately, a systematic bias appears and the RF predictor is not a good estimator of KRR predictor. The authors of \cite{jacot2020implicit} proved however that the average RF predictor $\E\br{ \hat f_\lambda^{(RF)}}$ is close to the KRR predictor $ \hat f_{\tilde \lambda}^{(K)}$ with a different parameter $\tilde \lambda$, which they called effective ridge parameter. In order to state their result, let us introduce $\gamma := \frac P N$ be the ratio of dimensions, $d_i$ the eigenvalues of $\cal K_X$, and the inverse kernel norm $\vertii{\bf v}_{\cal K_{\cal X}^{-1}} := \bf v^\top \cal K_{\cal X}^{-1} \bf v$.
\begin{thm}[\cite{jacot2020implicit},  Theorem 4.1] \label{JacotKernel} If $\Phi^{(j)}$ are Gaussian processes, for $N,P > 0$ and $\lambda > 0$, we have \[\verti{ \E\br{ \hat f_\lambda^{(RF)}(x)} - \hat f_{\tilde \lambda}^{(K)}(x)}  \leq \frac{c \sqrt{   \cal K(x,x)} \vertii{\cal Y}_{\cal K_{\cal X}^{-1}}} P,\]
where the effective ridge $\tilde \lambda > \lambda$ is the unique positive number satisfying
\[ \tilde \lambda = \lambda + \frac{\tilde \lambda} N \sum_{i=1}^N \frac{d_i}{\tilde \lambda + d_i},\]
and where $c>0$ depends on $\lambda, \gamma$, and $\frac 1 N \Tr \, \cal K_X$ only.    
\end{thm}

Let us define the measure $\check \nu$ as: \[\check \nu = \pt{1 - \frac N P} \delta_0 + \frac N P \pt{\mmp\pt{\frac N P}\boxtimes \mu_{\cal K_\cal X}}.\]
We will now see how our main result implies that $\E\br{ \hat f_\lambda^{(RF)}} \approx \hat f_{\tilde \lambda}^{(K)}$ with quantitative bounds, where $\tilde \lambda = g_{\check \nu}(\lambda)^{-1}$.

\begin{thm}\label{ApplicationKernel} Consider $\Phi^{(j)}$ random processes, not necessarily Gaussian, but such that the data matrix $F$ is $\propto \cal E_2(N^{-\half})$ concentrated. Assume that $\vertiii{\cal K_{\cal X}}$ is bounded, and that $\gamma$ is bounded from above and from below.

Then uniformly in any  $x \in \R^D$ and any bounded $\lambda > 0$ such that $\lambda^{-7} \leq O(N)$, as $N,P \to \infty $:
\[ \verti{ \E\br{ \hat f_\lambda^{(RF)}(x)} - \hat f_{\tilde \lambda}^{(K)}(x)} \leq O\pt{\frac{\cal K(x,x)^\half \vertii{\cal Y}_{\cal K_{\cal X}^{-1}} }{N^\half \lambda^{\frac 9 2}}  },\] where $\tilde \lambda = g_{\check \nu}(\lambda)^{-1} > \lambda$. The implicit constant in the $O(\cdot)$ notation only depends on the concentration constant of $F$, and the bounds on $\gamma$, $\cal K_{\cal X}$ and $\lambda$.
\end{thm}

\begin{rem} Both descriptions of the effective ridge parameter $\tilde \lambda$ in the Theorems  \ref{JacotKernel} and   \ref{ApplicationKernel}  are of course equivalent, a fact that is best seen using Proposition \ref{cfixedpoint}. As a consequence, one can easily retrieve all properties of $\tilde \lambda$ contained in the Proposition 4.2 of \cite{jacot2020implicit} using the general results on Stieltjes transforms and multiplicative free convolution.

One can also remark that the convergence speed in $N$ of Theorem \ref{JacotKernel} is faster. The authors relied on a similar deterministic equivalent result, in operator norm but with a better exponent than the one we obtained in Frobenius norm. This is made posible by reducing the problem to diagonal kernel matrices only, and we believe these arguments to be Gaussian specific.
\end{rem}

\begin{proof}  By assumption $P^\half F \propto \cal E_2(1)$, and the columns of $P^\half F$ are i.i.d., centered, and admit $\cal K_{\cal X}$ as covariance matrix. We can apply  Theorem \ref{MainThm} with  $z = - \lambda \in \R^{*-}$:
\begin{align*}
    \vertif{\E\br{\pt{{ FF^\top + \lambda I_N}}^{-1}}-\pt{\frac \lambda {\tilde \lambda}\cal K_{\cal X} + \lambda I_N }^{-1}} \leq O\pt{n^{-\half } \lambda^{- \frac {11} 2}}
\end{align*}
Given the identities = $\E\br{F\pt{{F^\top F + \lambda I_P}}^{-1} F^\top } = I_N - \lambda\pt{{ FF^\top + \lambda I_N}}^{-1} $ and $I_N - \lambda\pt{\frac \lambda {\tilde \lambda}\cal K_{\cal X} + \lambda I_N }^{-1} = \cal K_{\cal X}\pt{\cal K_{\cal X} + \tilde \lambda I_N}^{-1}$, we have~:
\begin{align*}
     \vertif{\E\br{F\pt{{F^\top F + \lambda I_P}}^{-1} F^\top } -\cal K_{\cal X}\pt{\cal K_{\cal X} + \tilde \lambda I_N}^{-1} } \leq O\pt{n^{-\half } \lambda^{- \frac {9} 2}}.
\end{align*}
We deduce that:
\begin{align*}
& \verti{   \E\br{ \hat f_\lambda^{(RF)}(x)} - \hat f_{\tilde \lambda}^{(K)}(x)} \\
    &= \verti{\cal K(x,\cal X) \pt{\cal K_{\cal X}^{-1} \E\br{F\pt{{F^\top F + \lambda I_P}}^{-1} F^\top }  - \pt{\cal K_{\cal X}+\tilde \lambda I_n}^{-1} } \cal Y } \\
    &\leq O\pt{N^{-\half } \lambda^{- \frac {9} 2}} \verti{\cal K(x,\cal X) \cal K_{\cal X}^{-1} \cal Y} \\
      &\leq O\pt{N^{-\half } \lambda^{- \frac {9} 2}} \vertii{\cal K(x,\cal X)}_{\cal K_{\cal X}^{-1}}^\half \vertii{\cal Y }_{\cal K_{\cal X}^{-1}}^\half,
\end{align*}
where we use the the Cauchy-Schwarz inequality on the inverse kernel norm to obtain the last estimate. As a general property of inverse Kernel norms, it can be shown that $\vertii{\cal K(x,\cal X)}_{\cal K_{\cal X}^{-1}}^2 \leq \cal K(x,x)$ (see the proof of Theorem C.8 in \cite{jacot2020implicit}). This concludes the proof.\end{proof}

\newpage

\printbibliography
\end{document}